%% file: main.tex
\documentclass[11pt]{article}
\usepackage{amsmath,amsfonts,amsthm,amssymb,color}
\usepackage{fancybox}
\usepackage{graphics}
\usepackage{graphicx}
\usepackage{natbib}

\usepackage{hyperref}

\newtheorem{theorem}{Theorem}[section]

\newtheorem{lemma}[theorem]{Lemma}

\newtheorem{claim}[theorem]{Claim}
\newtheorem{fact}[theorem]{Fact}

\theoremstyle{definition}

\newenvironment{fminipage}%
  {\begin{Sbox}\begin{minipage}}%
  {\end{minipage}\end{Sbox}\fbox{\TheSbox}}

\def\pleq{\preccurlyeq}

\def\defeq{\stackrel{\mathrm{def}}{=}}

\def\poly{\text{poly}}

\def\abs#1{\left|#1  \right|}

\def\norm#1{\left\| #1 \right\|}

\def\eps{\epsilon}

\newcommand\er{R_{\text{eff}}}

\DeclareMathOperator*{\argmin}{arg\,min}
\DeclareMathOperator*{\argmax}{arg\,max}

\def\tr{\text{tr}}

\def\im{\text{Im}}
\def\mss{the Marcus-Spielman-Srivastava theorem}

\title{Efficient Algorithms for Partitioning Circulant Graphs with Optimal Spectral Approximation}

\author{
    Surya Teja Gavva\thanks{Rutgers University, \href{mailto:suryateja.gavva@qc.cuny.edu}{suryateja@math.rutgers.edu}} \and 
    Peng Zhang\thanks{Rutgers University, \href{mailto:pz149@rutgers.edu}{pz149@rutgers.edu}}
}

\begin{document}

\date{} 

\clearpage
\maketitle
\thispagestyle{empty}

\begin{abstract}
  The Marcus-Spielman-Srivastava theorem (Annals of Mathematics, 2015) for the Kadison-Singer conjecture implies the following result in spectral graph theory: For any undirected graph $G = (V,E)$ with a maximum edge effective resistance at most $\alpha$, there exists a partition of its edge set $E$ into $E_1 \cup E_2$ such that the two edge-induced subgraphs of $G$ spectrally approximates $(1/2)G$ with a relative error $O(\sqrt{\alpha})$. However, the proof of this theorem is non-constructive. It remains an open question whether such a partition can be found in polynomial time, even for special classes of graphs.

    In this paper, we explore polynomial-time algorithms for partitioning circulant graphs via partitioning their generators.
    We develop an efficient algorithm that partitions a circulant graph whose generators form an arithmetic progression, with an error matching that in the Marcus-Spielman-Srivastava theorem and optimal, up to a constant.
    On the other hand, we prove that if the generators of a circulant graph are ``far" from an arithmetic progression, \emph{no} partition of the generators can yield two circulant subgraphs with an error matching that in the Marcus-Spielman-Srivastava theorem.
    
    In addition, we extend our algorithm to Cayley graphs whose generators are from a product of multiple arithmetic progressions.

\end{abstract}

\newpage
\pagenumbering{gobble}
\sloppy

\newpage

\pagenumbering{arabic}

\input{intro}

\input{prelims}

\input{circ_alg}

\input{lower}
\input{multi_dims}



\paragraph*{Funding Statement } The work of Peng Zhang was partially supported by the National Science Foundation (NSF) under research grant No. 2238682. The funder had no role in study design, data collection and analysis, decision to publish, or preparation of the manuscript.



\bibliographystyle{alpha}
\bibliography{refs}

\appendix

\input{appendix_miss_proofs}

\end{document}

%% file: intro.tex
\section{Introduction}
\label{sec:intro}

The Marcus-Spielman-Srivastava theorem (henceforth, the MSS Theorem) for the Kadison-Singer conjecture implies the following result for spectral graph theory:
\begin{theorem}[\cite{MSS15, srivastava13}]
Let $G = (V,E)$ be an undirected graph. Let $\alpha$ be the maximum effective resistance between 
any two endpoints of an edge in $G$.
Then, there \emph{exists} a partition of the edge set $E$ into $E_1 \cup E_2$ such that the two edge-induced subgraphs $G_1 = (V,E_1)$ and $G_2 = (V,E_2)$ satisfy
\begin{align}
\left( \frac{1}{2} - 5\sqrt{\alpha} \right) L_G \pleq L_{G_i} \pleq \left( \frac{1}{2} + 5\sqrt{\alpha} \right)  L_G, ~ \forall i \in \{1,2\} 
\label{eqn:mss}
\end{align}
where 
$L_G, L_{G_1}, L_{G_2}$ are the Laplacian matrices of graphs $G,G_1,G_2$ and 
$\pleq$ is the Loewner order for positive semidefinite matrices.
\label{thm:mss}
\end{theorem}

Partitioning the edge set $E$ of graph $G$ in Theorem \ref{thm:mss} uniformly at random yields 
two subgraphs $G_1$ and $G_2$ that each spectrally approximates  $(1/2)G$ with a relative error $O(\sqrt{\alpha \log n})$ with high probability \citep{tropp15}.
Theorem \ref{thm:mss} improves this error to $O(\sqrt{\alpha})$, opting to forego the ``high probability", achieving an optimal bound up to a constant.

In addition to its strong relevance to the Kadison-Singer conjecture, the task of
partitioning a graph into edge-disjoint subgraphs that preserves 
certain expansion properties have wide-ranging applications in network design \citep{BFU92,FM99}.
Furthermore, if $\alpha = O(1/d)$ where $d$ is the average vertex degree of $G$ -- a condition met by expander graphs and edge-transitive graphs,
applying Theorem \ref{thm:mss} iteratively to subgraphs $G_1$ and $G_2$ 
for $t = \Theta (\log d)$ levels results in $2^t$ spectral sparsifiers for $G$:
they are sparse subgraphs of $G$ that each spectrally approximates $(1/2^t)G$ with an error $O(1)$ \cite{srivastava13,paschalidis2024linear}.

Although Theorem \ref{thm:mss} is important, its proof is non-constructive, and it remains open whether a polynomial-time algorithm can find a partition satisfying Equation \eqref{eqn:mss} within a constant factor.
The currently best-known algorithm runs in $O(2^{\sqrt[3]{n} / \alpha})$ time, as developed by \cite{AGSS18} 
\footnote{The algorithm in \cite{AGSS18} applies to the more general form of \mss.}.
For graphs $G$ with strong expansion properties, 
\cite{FM99} developed a polynomial-time algorithm to partition \(G\) into subgraphs that maintain expansion properties within certain parameter ranges.
However, their bounds are far less tight than those in Theorem \ref{thm:mss}.
\cite{BCNPT20} designed an efficient algorithm 
to find a bounded-degree expander inside a \emph{very dense} $\Omega(\abs{V})$-regular expander.
\cite{BST19} generalized this result to any regular expanders by developing an algorithm for a weaker definition of graph sparsification with an \emph{additive} error.
As far as we are aware, we have not identified other polynomial-time algorithms for Theorem \ref{thm:mss} even when limiting to special classes of graphs.
In addition, finding an optimal partition in general is NP-hard \cite{spielman2022hardness,jourdan2023algorithmic}.

\subsection{Our Results}

We investigate efficient algorithms for Theorem \ref{thm:mss} applicable to special classes of circulant graphs, which are a special case of Cayley graphs. 
Cayley graphs play an important role in both combinatorial and geometric theory. 
Given a group \(\mathbb{G}\) and a subset $C \subset \mathbb{G}$ that is inverse-closed and excludes the identity element, 
the resulting Cayley graph \(X(\mathbb{G}, C)\) consists of vertices from \(\mathbb{G}\) and edges $(g,h)$ if \(hg^{-1} \in C\). 
The elements in \(C\) are called \emph{generators}.
A \emph{circulant graph} is a special type of Cayley graph of a cyclic group \(\mathbb{Z}_n\). 
For example, cycles, complete graphs, and Paley graphs are circulant graphs; hypercubes are Cayley graphs but not circulant graphs.

Our first result is a polynomial-time algorithm for circulant graphs whose generators
constitute an \emph{arithmetic progression}.

\begin{theorem}
Let $k,n$ be positive integers. Let $C = \{\pm a_1, \ldots, \pm a_k\} \pmod n \subseteq \mathbb{Z}_n \setminus \{0\}$
where $a_1, \ldots, a_k$ form an arithmetic progression.
Let $X = X(\mathbb{Z}_n, C)$ be the circulant graph.
Then, we can partition $C$ into symmetric \footnote{A set $S$ is \emph{symmetric} if $x \in S \iff -x \in S$.} subsets $C_1 \cup C_2$, in time polynomial in $k$, such that 
\begin{align*}
\left( \frac{1}{2} - O(1/\sqrt{k}) \right) L_{X} \pleq L_{X(\mathbb{Z}_n, C_i)} \pleq \left( \frac{1}{2} + O(1/\sqrt{k}) \right)  L_{X}, ~ \forall i \in \{1,2\}
\end{align*}
\label{thm:intro_ap}
\end{theorem}

Since the maximum effective resistance between any two endpoints of an edge in $X$ is at least \(\Omega(1/k)\), 
the error bound in Theorem \ref{thm:intro_ap} aligns with \mss \ (Theorem \ref{thm:mss}) within a constant factor. 
Additionally, we demonstrate that the error term \(O(1/\sqrt{k})\) is optimal within a constant factor; 
in other words, any partition of graph $X$ into two edge-induced subgraphs will necessarily result in an error term $\Omega(1/ \sqrt{k})$.
The runtime of the algorithm in Theorem \ref{thm:intro_ap} solely relies on \(k\), the number of generators, and is independent of \(n\), the number of vertices, which might be significantly greater than \(k\). 
On the other hand, verifying that a given partition satisfies Equation \eqref{eqn:mss} might require a time polynomial in terms of \(n\). 

Moreover, Theorem \ref{thm:intro_ap} can be viewed as a ``higher-rank" version of \mss, which might have additional advantages in specific applications \citep{branden18}.

Our second result is a lower bound for partitioning the generators of circulant graphs whose generators have large and irregular gaps between elements. 
For these types of graphs, different ideas will be needed to design algorithms that meet the conditions of \mss.

\begin{theorem}
There exist a constant $c > 0$, a constant $K_0 > 0$, and a function $h: \mathbb{Z}_+^{K+1} \rightarrow \mathbb{Z}_+$ such that the following holds for every integer $K > K_0$.
Given positive integers $a_1, \ldots, a_K$ satisfying 
\begin{align*}
\frac{a_k}{a_1 + \ldots + a_{k-1}} > 4 \log K, ~ \forall k = 2,\ldots,K
\end{align*}
and an integer $N > h(K,a_1,\ldots,a_K)$ such that the group $\mathbb{Z}_N$ can be generated by elements in $C = \{\pm a_1, \ldots, \pm a_K\}$,
let $X = X(\mathbb{Z}_N, C)$ the circulant graph.
Denote $\alpha$ as the maximum effective resistance between any two endpoints of an edge in $X$.
Then, \emph{no} partition of $C$ into symmetric subsets $C_1 \cup C_2$ satisfies
\[
\left( \frac{1}{2} - c \sqrt{\alpha \log K} \right) L_{X} \pleq L_{X(\mathbb{Z}_N, C_i)} \pleq 
\left( \frac{1}{2} + c \sqrt{\alpha \log K} \right)  L_{X}, ~ \forall i \in \{1,2\}     
\]
\label{thm:intro_lower}
\end{theorem}

Finally, we extend our algorithm in Theorem \ref{thm:intro_ap} to Cayley graphs of finite abelian groups, denoted by $\mathbb{Z}_{n_1} \times \cdots \times \mathbb{Z}_{n_d}$, 
with their generators being elements of a product of multiple arithmetic progressions. 

\subsection{Our Approaches}
\label{sec:intro_approach}

A key property of Cayley graphs is that the eigenvalues and eigenvectors of their graph Laplacian matrices have explicit forms.
Particularly, the eigenvectors are solely determined by the group \(\mathbb{G}\) that defines the graph. 
Consider partitioning a Cayley graph \(G\) via partitioning its generators, yielding two Cayley subgraphs \(G_1\) and \(G_2\) defined over the same group \(\mathbb{G}\). 
In this scenario, the Laplacian matrices \(L_G, L_{G_1}, L_{G_2}\) can be diagonalized simultaneously. 
Consequently, Equation \eqref{eqn:mss} from Theorem \ref{thm:mss} can be expressed as
\begin{multline*}
\left( \frac{1}{2} - 5\sqrt{\alpha} \right) \lambda_j(L_G) \le \lambda_j (L_{G_i}) \le \left( \frac{1}{2} + 5\sqrt{\alpha} \right) \lambda_j (L_G), 
\\ \forall i \in \{1,2\}, j \in \{1,\ldots,|\mathbb{G}|\} 
\end{multline*}
where \(\lambda_j(\cdot)\) refers to the \(j\)th eigenvalue.

\paragraph*{Algorithm for partitioning generators of circulant graphs}
We establish a connection between partitioning generators of circulant graphs and discrepancy theory \citep{matousek99,Chazelle00}.
The eigenvalues of the Laplacian matrix of a circulant graph $X = X(\mathbb{Z}_n, \{\pm a_1, \ldots, \pm a_k\})$ are 
\[
\lambda_j = 2 \sum_{s=1}^k 1 - \cos(a_s \theta_j), ~ \forall j \in \mathbb{Z}_n
\]
where $\theta_j = \frac{2\pi j}{n}$. This leads to the following equivalent formulation for the equation in Theorem \ref{thm:intro_ap},
\begin{align}
\abs{\sum_{s=1}^k x_s \left(1 - \cos(a_s \theta_j)\right)} \le O(1/\sqrt{k}) 
\sum_{s=1}^k ( 1 - \cos(a_s \theta_j)), ~ \forall j \in \mathbb{Z}_n
\label{eqn:intro_ap_goal}
\end{align}
where $x_s = 1$ if $\pm a_s \in C_1$ and $x_s = -1$ if $\pm a_s \in C_2$ for every $s \in \{1,\ldots,k\}$.
When the graph $X$ is an expander with $n= O(k)$, ensuring the right-hand side is $\Omega(\sqrt{k})$ for every $j \neq 0$,
signs $x_1, \ldots, x_k$ satisfying Equation \eqref{eqn:intro_ap_goal} exist and can be found in polynomial time via Spencer's theorem for discrepancy minimization \citep{spencer85,bansal10,LM15}.
Nonetheless, in general, the right-hand side of Equation \eqref{eqn:intro_ap_goal} can approach $0$ as every $a_s \theta_j$ tends to $0$ (modulo $2\pi$), and thus Spencer's theorem is not directly applicable.

To tackle the issue, we consider $\theta_j$'s with large right-hand side sums and those with small ones separately.
First, we reduce the problem of partitioning the generators from \emph{any} arithmetic progression to the one consisting of \(\{\pm 1, \ldots, \pm k\}\). 
Then, we divide the set of \(\theta_j\)'s into two subsets \(\Theta_1\) and \(\Theta_2\). The set $\Theta_1$ contains all $\theta_j$'s for which the right-hand side sum is $\Omega(k)$, and $\Theta_2$ the remaining $\theta_j$'s.
For each \(\theta_j\) in \(\Theta_1\), it suffices to find signs such that the left-hand side has absolute value \(O(\sqrt{k})\). However, a direct application of Spencer's theorem may result in a discrepancy of $\Omega(\sqrt{k \log (n/k)})$, 
which could be much larger than $\sqrt{k}$ when $n$ is much greater than $k$.
To address this, we adapt Spencer's method for signing complex-valued polynomials \citep{spencer85}: 
By Bernstein's inequality for trigonometric polynomials, it suffices to require only \(O(k)\) evenly-spaced values of \(\theta\) in \( (0, 2\pi)\) to satisfy the left-hand side's absolute value being \(O(\sqrt{k})\).
For each \(\theta_j\) in \(\Theta_2\), the value of $\theta_j \pmod{2 \pi}$ is sufficiently close to $0$, 
allowing us to approximate each $1 - \cos(a_s \theta_j)$ well using its Taylor series terms up to $O(\log k)$ terms.
We show that it is sufficient to have a total of \(O(\log k)\) terms with small absolute values to meet the conditions of Equation \eqref{eqn:intro_ap_goal} for all \(\theta_j\) in \(\Theta_2\).
Lastly, we employ Lovett and Meka's partial coloring theorem for minimizing discrepancy \citep{LM15} to find the signs meeting the conditions above. 
We remark our algorithm is capable of satisfying Equation \eqref{eqn:intro_ap_goal} for every \(\theta\) in \( (0,2\pi)\), 
which is a stronger claim that might be of independent interest.

\paragraph*{Lower bound for partitioning generators of circulant graphs}

Our proof for Theorem \ref{thm:intro_lower} consists of two components: 
the first demonstrates that no partition of the set \( C \) satisfies the condition in Theorem \ref{thm:intro_lower} 
when $c \sqrt{\alpha \log K}$ is replaced with $(1/40) \sqrt{(\log K) / K}$;
the second component establishes that the maximum effective resistance \(\alpha = O(1/K)\).
For the first component, we need to show that for any chosen signs \( x_1, \ldots, x_K \) representing a partition of \( C \),
\[
\max_{j\in\{1,\ldots,n-1\}}  \frac{\abs{\sum_{k=1}^K x_k \left(1 - \cos(a_k \theta_j)\right)}}{\sum_{k=1}^K 1 - \cos(a_k \theta_j)} > \frac{1}{20} \sqrt{\frac{\log K}{K}}.
\]
Given that the denominator in the equation is bounded above by \(2K\) for every $\theta_j$, it is sufficient to prove that the maximum of the numerator is at least \( (1/10) \sqrt{K \log K} \). 

We break down this numerator into two parts: the sum of the chosen signs $\sum_{k=1}^K x_k$ 
and the signed sum of the cosine functions $S_K(\theta_j) \defeq \sum_{k=1}^K x_k \cos(a_k \theta_j)$. 
Assume, by contradiction, there exist signs $x_1, \ldots, x_K$ that violate the above equation.
Then, the sum of these signs must have an absolute value no greater than \( (1/20) \sqrt{K \log K} \).
Our goal then becomes to prove \( \max_{j} \abs{S_K(\theta_j)} > (1/5) \sqrt{K \log K} \), thereby creating a contradiction. 
We extend the search for the maximum value of \( S_K(\theta) \) from the discrete set \(\{\frac{2\pi j}{N}: j=1,\ldots, N-1\}\) to the continuous interval \((0,2\pi)\) using Bernstein's inequality.
The \( S_K(\theta) \) where $a_1, \ldots, a_K$ have large gaps is an example of \emph{lacunary trigonometric polynomials}. 
The limit distribution of $S_K(\theta)$ when  $\theta$ is drawn uniformly from $(0,2\pi)$ and $K$ goes to infinity has been extensively studied \citep{PZ30,PZ30b,kac39,SZ47}.
However, these findings do not yield our desired bound for $S_K(\theta)$, which depends on \(K\).
To address this issue, we prove that $\abs{S_K(\theta)}$ is at least \((1/4) \sqrt{K \log K}\) with non-negligible probability when $\theta$ is drawn uniformly from $(0,2\pi)$, 
by employing an anti-concentration inequality on the \(2\lfloor \log K \rfloor\)th power of $S_K(\theta)$.

\subsection{Related Works}

\paragraph*{Weighted graph sparsifiers}

As discussed at the beginning of Section \ref{sec:intro}, 
applying \mss \  iteratively constructs \( O(1) \)-factor 
\emph{unweighted} spectral sparsifiers for \( (1/2^t) G \).

Efficient algorithms exist for constructing good spectral sparsifiers when \emph{reweighing edge weights} is allowed \citep{SS08,BSS09,ALO15,SL15, SL17}:
given any graph \( G = (V, E) \) and \( \eps > 0 \), the algorithms produce an edge-induced subgraph \( H = (V, F, w) \) where \( w: F \rightarrow \mathbb{R}_+ \), and the number of edges \( \abs{F} = O(\eps^{-2} \abs{V}) \) such that 
\[
(1 - \eps) L_G \pleq L_H \pleq (1 + \eps) L_G.
\]
However, it is unclear whether the approaches for generating weighted spectral sparsifiers can be transformed to create unweighted ones.

\paragraph*{Discrepancy minimization} 

Combinatorial discrepancy theory focuses on dividing a collection of elements --- such as numbers, vectors, or matrices --- into subsets that are as ``balanced'' as possible, improving on a uniformly random partition. 
Our work on partitioning the generators of circulant graphs reduces to partitioning a set of (potentially very high-dimensional) vectors to minimize a ``relative'' $\ell_{\infty}$ discrepancy.

A landmark result in this field is Spencer's theorem \citep{spencer85}, which asserts that given vectors \( v_1, \ldots, v_n \in [-1,1]^m \), one can find signs \( x_1, \ldots, x_n \in \{-1, 1\} \) such that \( \norm{\sum_{i=1}^n x_i v_i}_{\infty} = O(\sqrt{n \log (m/n+1)}) \). 
Although Spencer's original proof was non-constructive, subsequent research has yielded efficient algorithms that achieve the bounds set by this theorem \citep{bansal10,HSS14,LM15,rothvoss17,ES18,levy2017deterministic,bansal2022unified,pesenti2023discrepancy,JSS23}.
Better discrepancy bounds are known for vectors with additional structures,
for example, when the rows of the matrix \( \begin{pmatrix} v_1 & v_2 & \cdots & v_n \end{pmatrix} \) represent all the arithmetic progressions within \( \{1, \ldots, n\} \) \citep{roth1964remark,matouvsek1996discrepancy}
\footnote{This arithmetic progression structure is different from our problem setting.}.

\paragraph*{Trigonometric series} 

Littlewood investigated polynomials of the form \( p(z) = \sum_{i=0}^n \eps_i z^i \) where \( \eps_i \in \{\pm 1\} \) for all \( i = 0, \ldots, n \) and $z \in \mathbb{C}$. 
He asked whether there exist constants \( c_1, c_2 \) such that there exist \( \eps_1, \ldots, \eps_n \) meeting the criteria
$c_1 \sqrt{n+1} \le \max_{\abs{z} = 1} \abs{p(z)} \le c_2 \sqrt{n+1}$.
\cite{rudin59} and \cite{shapiro52} discovered \( \eps_1, \ldots, \eps_n \) that satisfy the upper bound with \( c_2 = \sqrt{2} \). 
Recently, \cite{BBMST20} confirmed Littlewood's query for both upper and lower bounds.
By Euler's formula, a polynomial $p(z)$ for complex $z$ can be written as a linear combination of $\cos(ix)$ and $\sin(ix)$ for $x \in \mathbb{R}$.
The study of trigonometric polynomials with randomly signed terms as the degree $n$ goes to infinity has been explored by \cite{SZ54}. Other works have also examined the convergence and limiting distributions of trigonometric series \citep{kac39,SZ47,erdos62,berkes78,fukuyama11}.

Our method for partitioning circulant graphs introduces a different question regarding signing trigonometric series, as outlined in Equation \eqref{eqn:intro_ap_goal}. 
This question could be interesting in its own right. 
There are two main differences between our question and prior research: 
(1) we focus on achieving a small \emph{relative} error rather than an absolute error,
and (2) the trigonometric series we consider, \( \sum_{s=1}^k x_s \cos(a_s \theta) \), allows terms with arbitrarily large degree $a_s$.

\paragraph*{Organization of the paper}
Section \ref{sec:prelim} outlines key notations and definitions.
Section \ref{sec:circ_ap} presents a polynomial-time algorithm for partitioning circulant graphs with arithmetic progression generators, proving Theorem \ref{thm:intro_ap}. 
Section \ref{sec:lower} establishes a lower bound for partitioning generators of circulant graphs and proves Theorem \ref{thm:intro_lower}. 
Section \ref{sec:multi_circ} extends the algorithm in Theorem \ref{thm:intro_ap} to Cayley graphs whose generators are from products of multiple arithmetic progressions.

%% file: prelims.tex
\section{Preliminaries}
\label{sec:prelim}

A matrix $A \in \mathbb{R}^{n \times n}$ is symmetric if $A = A^\top$ where $A^\top$ is the transpose of $A$.
A symmetric matrix $A$ is \emph{positive semidefinite} if $x^\top A x \ge 0$ for every $x \in \mathbb{R}^n$.
Let $B \in \mathbb{R}^{n \times n}$ be another symmetric matrix.
We say $A \pleq B$ if $B - A$ is positive semidefinite.
Every symmetric matrix $A$ admits a decomposition $A = \sum_{i=1}^r \lambda_i v_i v_i^\top$ where 
$\lambda_i$'s are the nonzero eigenvalues of $A$ and $v_i$'s are the corresponding orthonormal eigenvectors.
The \emph{pseudoinverse} of such a symmetric matrix $A$, denoted by $A^{\dagger}$, 
is $A^{\dagger} = \sum_{i=1}^r \lambda_i^{-1} v_i v_i^\top$.

For any positive integer $a$, let $[a] \defeq \{1,\ldots,a\}$. For any set $A \subseteq \mathbb{R}$, 
let $\pm A \defeq A \cup \{-a: a \in A\}$.

\subsection{Laplacian Matrices and Effective Resistances}

Given a graph $G = (V,E)$, the \emph{Laplacian matrix} of $G$, denoted by $L_G$, 
is defined as 
\[
L_G \defeq \sum_{(u,v) \in E} (\chi_u - \chi_v)  (\chi_u - \chi_v)^\top,
\]
where $\chi_u \in \mathbb{R}^{V}$ with value $1$ for entry $u$ and value $0$ everywhere else.
By definition, a Laplacian matrix is symmetric and positive semidefinite.
For any two vertices $u,v \in V$, the \emph{effective resistance} between $u,v$ in $G$, denoted by $\er(u,v)$,
is defined as 
\[
  \er(u,v) \defeq (\chi_u - \chi_v)^\top L_G^{\dagger} (\chi_u - \chi_v).  
\]
Treating $G$ as an electrical network where each edge is a resistor with resistance being the inverse of its weight,
$\er(u,v)$ can be interpreted as the effective resistance between $u$ and $v$ in the electrical network.
We include the proofs of the following facts in \ref{sec:appendix_prelim}.

\begin{fact}
Let $G = (V,E)$ be an unweighted graph with $n$ vertices and of average degree $d$. Then, 
\[
\max_{(u,v) \in E} \er(u,v) \ge \frac{2}{d} \left( 1 - \frac{1}{n}
\right).
\]
\label{fac:er_lower}
\end{fact}

\begin{fact}
Let $G$ be a $d$-regular graph. Then, \emph{no} partition of $G$ into two edge-induced subgraphs, denoted by $G_1$ and $G_2$, such that 
\[
    -\frac{1}{2\sqrt{d}} L_G \pleq  L_{G_1} - L_{G_2} \pleq \frac{1}{2\sqrt{d}} L_G.
\]
\label{fac:lower}
\end{fact}

\subsection{Circulant Graphs}

Let $\mathbb{Z}_n$ denote the additive group of integers modulo $n$, and let $C \subset \mathbb{Z}_n \setminus \{0\}$.
A \emph{circulant graph}, denoted $X = X(\mathbb{Z}_n, C)$, is a graph where the vertex set is $\mathbb{Z}_n$
and $(i,j)$ is an arc of $X$ if $i-j \in C$.
If $C$ is \emph{symmetric} (that is, $x \in C$ if and only if $-x \in C$), then the circulant graph $X(\mathbb{Z}_n, C)$ is \emph{undirected}.
The elements in $C$ are called the \emph{generators} of $X$. 
Examples of circulant graphs include cycles, complete graphs, and Paley graphs. A circulant graph \(X(\mathbb{Z}_n, C)\) is \(\abs{C}\)-regular by definition.

Eigenvectors and eigenvalues of the Laplacian matrix of a circulant graph have explicit forms. In particular, the eigenvectors only depend on the number of vertices in the graph.

\begin{fact}
Let $X = X(\mathbb{Z}_n, \{\pm a_1, \ldots, \pm a_k\})$ be a circulant graph. Then, the eigenvalues and the corresponding orthonormal eigenvectors of the Laplacian matrix $L_X$ of $X$ are
\begin{align*}
\lambda_j = 2 \sum_{s=1}^k \left( 1 - \cos(a_s \frac{2\pi j}{n}) \right), ~
v_j = \frac{1}{\sqrt{n}} \left( 1, \omega^j, \omega^{2j}, \ldots, \omega^{(n-1)j}
\right)^\top, ~
\forall j \in \mathbb{Z}_n
\end{align*}
where $\omega = \exp(\frac{2\pi i}{n})$ and $i = \sqrt{-1}$.
\label{fac:circ_eigs}
\end{fact}

%% file: circ_alg.tex
\section{An Efficient Algorithm for Partitioning Circulant Graphs 
Whose Generators Form an Arithmetic Progression}
\label{sec:circ_ap}

In this section, we describe an efficient algorithm for partitioning a circulant graph 
whose generators constitute an arithmetic progression, and we prove Theorem \ref{thm:intro_ap}.
By Fact \ref{fac:lower}, the error term $O(1/\sqrt{k})$ in Theorem \ref{thm:intro_ap} is optimal up to a constant.

\begin{theorem}[Restatement of Theorem \ref{thm:intro_ap}]
\label{thm:circ_ap}
Let $n,b,k \in \mathbb{Z}_{>0}$ and $a \in \mathbb{Z}_{\ge 0}$. 
Let 
\begin{align*}
C =  \left\{ \pm (a+b), \pm (a+2b), \ldots, \pm (a+kb)
\right\} \hspace*{-.3cm} \pmod n \subseteq \mathbb{Z}_n \setminus \{0\}.
\end{align*}
Let $X = X(\mathbb{Z}_n, C)$ be the circulant graph. 
Then, we can partition $C$ into two symmetric subsets $C_1 \cup C_2$, in time polynomial in $k$, such that 
\[
-O(\frac{1}{\sqrt{k}}) L_{X} \pleq
L_{X(\mathbb{Z}_n,C_1)} - L_{X(\mathbb{Z}_n,C_2)} \pleq O(\frac{1}{\sqrt{k}}) L_{X}.
\]
\end{theorem}

Without loss of generality, we assume in Theorem \ref{thm:circ_ap} the elements in \(C\) generate the group \(\mathbb{Z}_n\), that is, graph $X$ is connected. Otherwise, we can deal with each connected component of $X$ separately.

To prove Theorem \ref{thm:circ_ap}, as we discussed in Section \ref{sec:intro_approach}, 
it suffices to find $x_1, \ldots, x_k \in \{\pm 1\}$ satisfying
\begin{align}
\max_{\theta \in \Theta} \abs{\frac{\sum_{s=1}^k x_s \left( 1 - \cos( (a+sb)\theta)
\right)}{\sum_{s=1}^k  1 - \cos( (a+sb) \theta)}}
= O(\frac{1}{\sqrt{k}}),
\label{eqn:circ_ap_goal}
\end{align}
where $\Theta = \left\{ \frac{2\pi j}{n}: j = 1,\ldots,n-1 \right\}$. 
We split $\Theta$ into two subsets: Let $\alpha = 0.9$,
\begin{align}
\begin{split}
& \Theta_1 \defeq \left\{\theta \in \Theta: b \theta \in \left(\frac{\alpha}{k}, 2\pi - \frac{\alpha}{k} \right) + 2 \pi l, \text{ for some integer } l \right\} \\
& \Theta_2 \defeq \Theta \setminus \Theta_1.
\end{split}
\label{eqn:theta1_def}
\end{align}
We will deal with the two subsets separately to derive sufficient conditions for the signs $x_1, \ldots, x_k \in \{\pm 1\}$
such that Equation \eqref{eqn:circ_ap_goal} holds and signs satisfying these conditions can be 
constructed efficiently.

\subsection{Sufficient Conditions for Set $\Theta_1$}

We derive sufficient conditions for the signs $x_1, \ldots, x_k$ such that Equation \eqref{eqn:circ_ap_goal}
holds for every $\theta \in \Theta_1$.

The denominator in the left-hand side of Equation \eqref{eqn:circ_ap_goal} is $k - \sum_{s=1}^k \cos((a+sb)\theta) \ge 0$.
We exploit the facts: (1) $\cos((a+sb) \theta)$ equals the real part of $e^{i(a+sb)\theta}$ where $i = \sqrt{-1}$; 
and (2) $\sum_{s=1}^k e^{s(i b\theta)}$ is the sum of a geometric sequence.
\begin{multline}
\sum_{s=1}^k \cos((a+sb) \theta)
\le \abs{ \sum_{s=1}^k e^{i(a+sb)\theta}} 
\le \abs{ \sum_{s=1}^k e^{isb\theta}} \\
\le \sqrt{ \left( \sum_{s=1}^k \cos(sb\theta) \right)^2 + \left( \sum_{s=1}^k \sin(sb\theta) \right)^2 }.
\label{eqn:fkt_def}
\end{multline}

The following lemma lower bounds the denominator in Equation \eqref{eqn:circ_ap_goal}. We defer its proof to  \ref{sec:appendix_circ_ap}. 

\begin{lemma}
For every $\theta \in \Theta_1$, $\sum_{s=1}^k 1 - \cos((a+sb) \theta) \ge \frac{\alpha^2 k}{96} = \Omega(k)$.
\label{lem:circ_mid_den}
\end{lemma}

Lemma \ref{lem:circ_mid_den}
implies that to achieve Equation \eqref{eqn:circ_ap_goal} for $\Theta_1$, it suffices to find signs $x_1, \ldots, x_k \in \{\pm 1\}$ such that the numerator 
$$\abs{\sum_{s=1}^k x_s (1-\cos((a + sb)\theta))} = O(\sqrt{k}), ~ \forall \theta \in \Theta_1.$$
We apply Spencer's idea of reducing the $\Omega(\abs{\Theta_1})$ inequalities to $O(k)$ inequalities via Bernstein's inequality for trigonometric polynomials \citep{spencer85}.
Note that $\abs{\Theta_1} = \Theta(n)$ which can be much larger than $k$.

\begin{theorem}[Bernstein's inequality]
Let $T(x) = \sum_{s=-k}^k a_s e^{isx}$ be a trigonometric polynomial of degree $k$. Then, 
$\sup_{x \in \mathbb{R}} \abs{T'(x)} \le k \sup_{x \in \mathbb{R}} \abs{T(x)}$,
and the constant $k$ is optimal in general (e.g., $T(x) = e^{ikx}$).
\label{thm:bernstein}
\end{theorem}

We want to write the numerator as a trigonometric polynomial of a small degree. Since $a$ and $b$ can be large, we want to pull out the common angle $a\theta$ in each cosine and combine $b$ with $\theta$.
We write the cosine terms as the real parts of $e^{i(a + sb)\theta}$.
Since $\theta \in \Theta_1$, we can write $b\theta = \hat{\theta} + 2\pi l$ where $\hat{\theta} \in \left(\frac{\alpha}{k}, 2\pi - \frac{\alpha}{k}
\right)$ and $l$ is an integer.
\begin{multline}
\abs{\sum_{s=1}^k x_s (1 - \cos((a+sb)\theta))}
\le \abs{\sum_{s=1}^k x_s} + \abs{\sum_{s=1}^k x_s e^{i(a\theta + s\hat{\theta})}} \\
\le \abs{\sum_{s=1}^k x_s} + \abs{\sum_{s=1}^k x_s e^{is\hat{\theta}}}
\le \abs{\sum_{s=1}^k x_s} + \abs{\sum_{s=1}^k x_s \cos(s \hat{\theta})}
+ \abs{\sum_{s=1}^k x_s  \sin(s \hat{\theta})}.
\label{eqn:circ_approx_num}
\end{multline}
If we set $x_0 = 0$ and $x_{-s} = x_s$ for $s=1,\ldots,k$, 
then $\sum_{s=1}^k x_s \cos(s\hat{\theta}) = \frac{1}{2} \sum_{s=-k}^k x_s e^{is\hat{\theta}}$ is a trigonometric polynomial in $\hat{\theta}$ of degree $k$.
Similarly, if we set $x_0 = 0$ and $x_{-s} = -x_s$ for $s=1,\ldots,k$, 
then $\sum_{s=1}^k x_s \sin(s\hat{\theta}) = -\frac{i}{2} \sum_{s=-k}^k x_s e^{is\hat{\theta}}$ is a trigonometric polynomial in $\hat{\theta}$ of degree $k$.
The following lemma states that it suffices to find signs so that the rightmost-hand side of Equation \eqref{eqn:circ_approx_num} is at most $O(\sqrt{k})$ for $O(k)$ different values of $\hat{\theta}$.

\begin{lemma}
Let $\Lambda \defeq \left\{ \frac{2\pi j}{7k}: j = 0, 1,2,\ldots,7k-1 \right\}$.
Suppose $x_1, \ldots, x_k$ satisfies 
\begin{align}
\abs{\sum_{s=1}^k x_s \cos(s\hat{\theta})}, \abs{\sum_{s=1}^k x_s \sin(s\hat{\theta})} = O(\sqrt{k}), ~ \forall \hat{\theta} \in \Lambda.
\label{eqn:circulant_middle_signs}
\end{align}
Then, Equation \eqref{eqn:circ_ap_goal} holds for every $\theta \in \Theta_1$.
\label{lem:circ_cond_large_den}
\end{lemma}

\begin{proof}
Suppose $x_1, \ldots, x_k \in \{\pm 1\}$ satisfy Equation \eqref{eqn:circulant_middle_signs}.
Then, when $\hat{\theta} = 0$, we have $\abs{\sum_{s=1}^k x_s} = O(\sqrt{k})$.

We will show $\abs{\sum_{s=1}^k x_s \cos(s \hat{\theta})}, \abs{\sum_{s=1}^k x_s  \sin(s \hat{\theta})} = O(\sqrt{k})$ for every $\hat{\theta} \in (0,2\pi)$.
We extend our signs by setting $x_0 = 0$ and $x_{-s} = x_s$ for $s=1,\ldots,k$. Let
\[
T(\hat{\theta}) \defeq \frac{1}{2} \sum_{s=-k}^k x_s e^{is \hat{\theta}}, ~
\hat{\theta}^* \defeq \argmax_{\theta \in (0,2\pi)} \abs{T(\hat{\theta})}.
\]
Then, $T(\hat{\theta}) = \sum_{s=1}^k x_s \cos(s\hat{\theta})$.
There exists a $\theta \in \Lambda$ so that $\abs{\hat{\theta}^* - \hat{\theta}} < \frac{2\pi}{7k}$. By the mean value theorem and Bernstein's inequality (Theorem \ref{thm:bernstein}), 
\begin{align*}
\abs{T(\hat{\theta}^*)} - \abs{T(\hat{\theta})} \le 
\abs{T(\hat{\theta}^*) - T(\hat{\theta})}  \le \frac{2\pi}{7k} \cdot \max_{\hat{\theta}' \in (0,2\pi)} \abs{T'(\hat{\theta}')}
\le \frac{2\pi}{7} \abs{T(\hat{\theta}^*)}.
\end{align*}
By the assumption in Equation \eqref{eqn:circulant_middle_signs},
\[
\abs{T(\hat{\theta}^*)} \le (1-\frac{2\pi}{7}) \abs{T(\hat{\theta})} = O(\sqrt{k}).
\]
Similarly, the assumption in Equation \eqref{eqn:circulant_middle_signs} for the signed sum of the sine functions implies $\abs{\sum_{s=1}^k x_s \sin(s\hat{\theta})} = O(\sqrt{k})$ for every $\theta \in (0,2\pi)$.
Plugging into Equation \eqref{eqn:circ_approx_num}, we have $\abs{\sum_{s=1}^k x_s \left( 1 - \cos((a+sb)\theta)
\right)} = O(\sqrt{k})$, which upper bounds the numerator in Equation \eqref{eqn:circ_ap_goal}.
By the lower bound of the denominator for $\theta \in \Theta_1$ in Lemma \ref{lem:circ_mid_den}, Equation \eqref{eqn:circ_ap_goal} holds for every $\theta \in \Theta_1$.
\end{proof}

\subsection{Sufficient Conditions for Set $\Theta_2$}

We derive sufficient conditions for the signs $x_1, \ldots, x_k$ such that Equation \eqref{eqn:circ_ap_goal}
holds for every $\theta \in \Theta_2$.
For $\theta \in \Theta_2$, the denominator value of Equation \eqref{eqn:circ_ap_goal} can be much smaller than $k$.
Since $b\theta$ (modulo $2\pi$) is sufficiently close to $0$, we will approximate the cosine and sine functions via their Taylor expansion approximation. 

We write $b\theta = \hat{\theta} + 2\pi l$ where $\hat{\theta} \in \left[-\frac{\alpha}{k}, \frac{\alpha}{k} \right]$ and $l$ is an integer (recall $\alpha = 0.9$).
Then,
\begin{align*}
& \sum_{s=1}^k x_s (1 - \cos((a+sb)\theta)) 
 \\ = & \cos(a\theta) \sum_{s=1}^k x_s (1-\cos(s \hat{\theta}))
+ \sin(a\theta) \sum_{s=1}^k x_s \sin(s \hat{\theta})
+ (1-\cos(a\theta)) \sum_{s=1}^k x_s \\
\defeq & \cos(a\theta) \beta_1 + \sin(a\theta) \beta_2 + (1-\cos(a\theta)) \beta_3 \\
\defeq & g(\beta_1, \beta_2, \beta_3).
\end{align*}
In addition, we define $\beta_j' = \beta_j$ with $x_1 = \cdots = x_k = 1$ for $j=1,2,3$.
Note that the left-hand side of Equation \eqref{eqn:circ_ap_goal} equals $\abs{g(\beta_1, \beta_2, \beta_3)/ g(\beta_1', \beta_2', \beta_3')}$.
The following claim upper bounds this ratio via upper bounding the ratios $\abs{\beta_j} /\abs{ \beta_j'}$ for $j=1,2,3$.
We defer its proof to \ref{sec:appendix_circ_ap}.

\begin{claim}
Suppose $\abs{\beta_j} \le \gamma \abs{\beta'_j}$ for $j =1,2,3$. Then, 
\begin{align*}
\abs{\frac{g(\beta_1, \beta_2, \beta_3)}{ g(\beta_1', \beta_2', \beta_3')}} \le 13 \gamma.
\end{align*}
\label{clm:circ_small_den_goal}
\end{claim}

To guarantee Equation \eqref{eqn:circ_ap_goal} for $\Theta_2$, by Claim \ref{clm:circ_small_den_goal}, 
it suffices to find signs $x_1, \ldots, x_k \in \{\pm 1\}$ such that $\abs{\beta_j} \le O(\frac{1}{\sqrt{k}}) \abs{\beta_j'}$ for $j=1,2,3$.
The following lemma further simplifies these conditions.

\begin{lemma}
Let $L \defeq 38\log k$. Suppose the signs $x_1, \ldots, x_k$ satisfies 
\begin{align}
\abs{\sum_{s=1}^k x_s \left( \frac{s}{k}
\right)^l} = O(\log k), ~ \forall 0 \le l \le L.
\label{eqn:circ_cond_small_den}
\end{align}
Then, Equation \eqref{eqn:circ_ap_goal} holds for every $\theta \in \Theta_2$.
\label{lem:circ_cond_small_den}
\end{lemma}

\begin{proof}
The assumption in Equation \eqref{eqn:circ_cond_small_den} with $l=0$ implies 
\[
\abs{\beta_3} \le O(\frac{\log k}{k}) \beta_3' = O(\frac{1}{\sqrt{k}})\beta_3'.
\]
For $\beta_1$ and $\beta_1'$, since $s\hat{\theta} \in \left[-\alpha, \alpha\right]$, we approximate 
\begin{align*}
\beta_1' = 
\sum_{s=1}^k 1- \cos(s\hat{\theta}) \ge \sum_{s=1}^k \frac{(s\hat{\theta})^2}{2} - \frac{(s\hat{\theta})^4}{4!}
\ge \frac{1}{2}(1-\frac{\alpha^2}{12}) \hat{\theta}^2 \sum_{s=1}^k s^2 \ge \frac{1}{12} k^3 \hat{\theta}^2.
\end{align*}
For $\beta_1$, we take the Taylor expansion for $\cos (s\hat{\theta})$ and apply Taylor's reminder theorem.
\begin{align*}
\beta_1 = 
\abs{\sum_{s=1}^k x_s (1-\cos(s\hat{\theta}))} 
\le \sum_{\substack{2 \le l\le L \\ l \text{ is even}}} \abs{\sum_{s=1}^k x_s \frac{(s\hat{\theta})^{l}}{{l}!} } 
+ \abs{\sum_{s=1}^k \frac{(s\hat{\theta})^{L+2}}{(L+ 2)!} }.
\end{align*}
The last term
\begin{align*}
 \sum_{s=1}^k \frac{(s\hat{\theta})^{L+2}}{(L+2)!}  
 \le \hat{\theta}^2 k^3 (k \hat{\theta})^{L}
 \le \hat{\theta}^2 k^3 \alpha^{L}
 \le \frac{\hat{\theta}^2}{k},
\end{align*}
where the second inequality holds due to $-\frac{\alpha}{k} \le \hat{\theta} \le \frac{\alpha}{k}$.
Thus,
\begin{align*}
\frac{\abs{\beta_1}}{\beta_1'}
& \le 12 k^{-3} \sum_{\substack{2 \le l \le L \\ l \text{ is even}} } \frac{\hat{\theta}^{l-2}}{l!} \abs{\sum_{s=1}^k x_s s^l } + O(k^{-4}) \\
& \le 12 k^{-1} \sum_{\substack{2 \le l \le L \\ l \text{ is even}} } \abs{\sum_{s=1}^k x_s \left( \frac{s}{k} \right)^l } + O(k^{-4})
\tag{since $\hat{\theta} < \frac{1}{k}$} \\
& = O(k^{-1} \log^2 k)
\tag{by Lemma assumption in Equation \eqref{eqn:circ_cond_small_den}} \\
& = O(\frac{1}{\sqrt{k}}).
\end{align*}
Similarly, we approximate 
\begin{align*}
\abs{\beta_2'} \ge \frac{1}{4}k^2 \abs{\hat{\theta}},
~ \abs{\beta_2} \le \sum_{\substack{1 \le l < L \\ l \text{ is odd}}} \abs{\sum_{s=1}^k x_s \frac{(s\hat{\theta})^l}{l!}} + O(\frac{\hat{\theta}^2}{k}).
\end{align*}
Equation \eqref{eqn:circ_cond_small_den} guarantees $\abs{\beta_2} \le O(\frac{1}{\sqrt{k}}) \abs{\beta_2'}$.
By Claim \ref{clm:circ_small_den_goal}, Equation \eqref{eqn:circ_ap_goal} holds for $\Theta_2$.
\end{proof}

\subsection{Constructing Signs}
\label{sec:circ_ap_combine}

We construct signs $x_1, \ldots, x_k \in \{\pm 1\}$ that satisfy the conditions in Lemma \ref{lem:circ_cond_large_den} and \ref{lem:circ_cond_small_den}, which implies Equation \eqref{eqn:circ_ap_goal} holds for every $\theta \in (0,2\pi)$ and proves Theorem \ref{thm:circ_ap}.
We will iteratively apply the following partial signing theorem by \cite{LM15}.

\begin{theorem}[\cite{LM15}]
Let $v_1, \ldots, v_m \in \mathbb{R}^n$ and $x_0 \in [-1,1]^n$. 
Let $c_1, \ldots, c_m \ge 0$ be thresholds such that $\sum_{j=1}^m \exp(-c_j^2 / 16) \le n / 16$. 
Let $\delta > 0$ be a small parameter. Then, there exists a randomized algorithm that finds an $x \in \{\pm 1\}^n$ such that with probability at least $0.1$,
(1) $\abs{\langle v_j, x-x_0 \rangle} \le c_j \norm{v_j}_2$, for every $j \in \{1,\ldots,m\}$;
and (2) $\abs{x_j} \ge 1 - \delta$ for at least $n/2$ indices $j \in \{1,\ldots,n\}$.
Moreover, the algorithm runs in time 
$O((m+n)^3 \delta^{-2} \log (nm/\delta))$.
\label{thm:LM15}
\end{theorem}

The success probability $0.1$ in Theorem \ref{thm:LM15} can be boosted to $\frac{1}{\poly(nm)}$ by independently repeating the algorithms $\Theta (\log (mn))$ times.

\begin{proof}[Proof of Theorem \ref{thm:circ_ap}]
    
We construct $m \defeq 14k + L + 1 = O(k)$ vectors in $k$ dimensions, which correspond to the conditions in Lemma \ref{lem:circ_cond_large_den} and \ref{lem:circ_cond_small_den}. 
We define for $1 \le j \le 7k$,
\[
v_j \defeq \begin{pmatrix}
    \cos(\theta_j), & \cos(2\theta_j), & \ldots, & \cos(k\theta_j)
\end{pmatrix}^\top,
\]
where $\theta_j = \frac{2\pi (j-1)}{7k}$, for $7k+1 \le j \le 14k$, 
\[
v_j \defeq \begin{pmatrix}
    \sin(\theta_{j-7k}), & \sin(2\theta_{j-7k}), & \ldots, & \sin(k\theta_{j-7k})
\end{pmatrix}^\top,
\]
and for $j \ge 14k +1$,
\[
v_j \defeq \begin{pmatrix}
    (\frac{1}{k})^{(j-14k-1)}, & (\frac{2}{k})^{(j-14k-1)}, & \ldots, & (\frac{k}{k})^{(j-14k-1)}    
\end{pmatrix}^\top.
\]
In addition, we set $\delta = \frac{1}{k}$.

We iteratively apply Theorem \ref{thm:LM15} with these parameters.
Let $x^{(t)}$ be the signing vector at iteration $t$.
Initially, $x^{(0)} = 0$. 
For $t = 1, 2, \ldots$, we say variable $s \in \{1,\ldots,k\}$ is \emph{alive} if $\abs{x^{(t-1)}_s} < 1-\delta$.
Let $a_t$ be the number of alive variables. If $a_t \le 32(L+1)$, we terminate the iterations. 
Otherwise, we apply Theorem \ref{thm:LM15} to the sub-vectors of $v_1, \ldots, v_m$ restricted to alive variables. 
We set the thresholds $c_j = 7 \sqrt{a_t \ln(14k / a_t)}$ for $1 \le j \le 14k$ and $c_j = 0$ for $j \ge 14k + 1$. We can check these thresholds satisfy the condition in Theorem \ref{thm:LM15}:
\begin{align*}
\sum_{j=1}^{14k} \exp \left( - \frac{47 \ln(14k / a_t)}{16}
\right) + L + 1
\le 14k \cdot (\frac{a_t}{14k})^{16/47} + \frac{a_t}{32}
\le \frac{a_t}{16}.
\end{align*}
The last inequality holds since $a_t \le k$.
By Theorem \ref{thm:LM15} with $x_0 = x^{(t-1)}$, with high probability and in time polynomial in $k$, we can find $x^{(t)} \in [-1,1]^k$ such that (1) $x^{(t)}$ has at most $\frac{a_t}{2}$ alive variables and $x^{(t)}_s = x^{(t-1)}_s$ if $\abs{x^{(t-1)}_s} \ge 1-\delta$, and (2)
\[
\abs{\langle v_j, x^{(t)} - x^{(t-1)} \rangle} \le c_j \sqrt{a_t}, ~ \forall j =1,\ldots,m
\]
When we terminate the algorithm at iteration $T$, for every $1 \le j \le 14k$,
\begin{multline*}
\abs{\langle v_j, x^{(T)} \rangle}
\le \sum_{t} 7 \sqrt{a_t \ln (\frac{14k}{a_t}) } 
\le 7 \sqrt{k} \sum_{t} \sqrt{\frac{a_t}{k} \ln (\frac{14k}{a_t}) } \\
\le 7 \sqrt{k} \sum_t \sqrt{\frac{ \ln(14 \cdot 2^{t-1})} {2^{t-1}}} = O(\sqrt{k}),
\end{multline*}
and for every $j \ge 14k+1$, 
\[
\langle v_j, x^{(T)} \rangle = 0.
\]
In addition, at most $O(\log k)$ variables are still alive.
We arbitrarily sign these alive variables $\pm 1$, which increases the discrepancy at most $O(\log k)$. 
We round every $x_s$ with an absolute value less than $1$ to its sign, which increases the discrepancy by $O(1)$ since the entries of $v_j$'s are in $[-1,1]$. 
So, we get signs $x_1, \ldots, x_k \in \{\pm 1\}$ satisfying the conditions in Lemma \ref{lem:circ_cond_large_den} and \ref{lem:circ_cond_small_den} and establish Theorem \ref{thm:circ_ap}.
\end{proof}

%% file: lower.tex
\section{A Lower Bounds for Partitioning the Generators of Circulant Graphs}
\label{sec:lower}

In this section, we establish a lower bound for the error of partitioning the generators of a circulant graph with increasingly large gaps between them, proving Theorem \ref{thm:intro_lower}. 
First, we prove a lower bound for the approximation error in terms of \(K\), the number of generators. 
Second, we show that the maximum effective resistance between any two endpoints of an edge in a circulant graph
is \(O\left(\frac{1}{K}\right)\).

\begin{theorem}
There exists a constant $K_0 > 0$ such that the following holds for every integer $K > K_0$.
Let $a_1, \ldots, a_K$ be positive integers satisfying 
\begin{align}
\frac{a_k}{a_1 + \ldots + a_{k-1}} > 4 \log K, ~ \forall k = 2,\ldots,K,
\label{eqn:lower_assumption}
\end{align}
and let $N \ge 32 a_K$ be an integer. In addition, we assume $C = \{\pm a_1, \ldots, \pm a_K\}$ generates the group $\mathbb{Z}_N$.
Let $X = X(\mathbb{Z}_N, C)$ be the circulant graph.
Then, \emph{no} partition of $C$ into symmetric subsets $C_1 \cup C_2$ satisfies
\begin{align}
- \frac{1}{20} \sqrt{\frac{\log K}{K}} L_{X} \pleq
L_{X(\mathbb{Z}_N, C_1)} - L_{X(\mathbb{Z}_N, C_2)} \pleq \frac{1}{20} \sqrt{\frac{\log K}{K}}  L_{X}.
\label{eqn:lower_matrix_goal}
\end{align}
\label{thm:lower}
\end{theorem}

Theorem \ref{thm:lower} is equivalent to say that for every integer $K > K_0$, any signs $x_1, \ldots, x_K \in \{\pm 1\}$ (indicating the partition $C_1 \cup C_2$) must satisfy
\[
\max_{\theta \in \Theta_N} \abs{\frac{\sum_{k=1}^K x_k (1 - \cos(a_k \theta))}{\sum_{k=1}^K (1 - \cos(a_k \theta))}} >  \frac{1}{20} \sqrt{\frac{\log K}{K}},
\]
where recall $\Theta_N = \{\frac{2\pi j}{N}: j=1,\ldots,N-1 \}$.
Given that $\sum_{k=1}^K (1 - \cos(a_k \theta)) \le 2K$ for any $\theta$, it suffices to show 
\begin{align}
\max_{\theta \in \Theta_N} \abs{\sum_{k=1}^K x_k (1 - \cos(a_k \theta))} > \frac{1}{10} \sqrt{K \log K}.
\label{eqn:lower_numerator}
\end{align}
We separate the sum of the signs and the signed sum of cosine functions:
\begin{align*}
\sum_{k=1}^K x_k (1-\cos(a_k\theta))
= \sum_{k=1}^K x_k - \sum_{k=1}^K x_k \cos(a_k\theta).
\end{align*}
By the definition of circulant graphs, $\sum_{k=1}^K x_k$ equals the difference between the vertex degree of $X(\mathbb{Z}_N, C_1)$ and that of $X(\mathbb{Z}_N, C_2)$.
This difference can be bounded by the following claim.

\begin{claim}
Suppose $G$ is a $K$-regular graph, and $G_1, G_2$ are two edge-induced subgraphs of $G$ satisfying
\[
-\gamma L_G \pleq L_{G_1} - L_{G_2} \pleq \gamma L_G.
\]
Then, for any vertex $v$ in $G$, the difference between the degree of $v$ in $G_1$ and that in $G_2$ must be 
between $-\gamma K$ and $\gamma K$.
\label{clm:lower_degree}
\end{claim}
\begin{proof}
Let $\chi_v$ be the $v$th elementary basis vector, where the $v$th entry is $1$ and all others are $0$. 
Then,
\[
-\gamma \cdot \chi_v^\top L_G \chi_v 
\le \chi_v^\top (L_{G_1} - L_{G_2}) \chi_v
\le \gamma \cdot \chi_v^\top L_G \chi_v. 
\]
For any graph $G$, the value of
$\chi_v^\top L_G \chi_v$ equals the degree of $v$ in $G$. Thus, the above equation implies the difference between the degree of $v$ in $G_1$ and that in $G_2$ is between $-\gamma K$ and $\gamma K$.
\end{proof}

Assume, by contradiction, there exist signs $x_1, \ldots, x_K$ satisfying Equation \eqref{eqn:lower_matrix_goal}.
By Claim \ref{clm:lower_degree}, $\abs{\sum_{k=1}^K x_k} \le \frac{1}{20} \sqrt{K \log K}$.
Suppose we can show that for any $x_1, \ldots, x_K \in \{\pm 1\}$, 
\begin{align}
\max_{\theta \in \Theta_N} \abs{\sum_{k=1}^K x_k \cos(a_k\theta)} \ge \frac{1}{5} \sqrt{K \log K}.
\label{eqn:lower_numerator_split}
\end{align}
Then, by the triangle inequality, Equation \eqref{eqn:lower_numerator} holds, leading to a contradiction of the assumption.

To prove Equation \eqref{eqn:lower_numerator_split} holds for any $x_1, \ldots, x_K \in \{\pm 1\}$,
we denote 
\[
S_K(x_1, \ldots, x_K, \theta) \defeq \sum_{k=1}^K x_k \cos(a_k \theta), 
~ \forall x_1, \ldots, x_K \in \{\pm 1\}
\]
When the context is clear, we write $S_K(\theta) = S_K(x_1, \ldots, x_K, \theta)$.
The sum $S_K(\theta)$ is a special case of \emph{lacunary trigonometric series} \citep{kac39,SZ47}.

\subsection{Lacunary Trigonometric Series}

To prove Theorem \ref{thm:lower},
we will first establish a lower bound for \\
$\max_{\theta \in (0,2\pi)} \abs{S_K(\theta)}$
for every fixed $x_1, \ldots, x_K \in \{\pm 1\}$ and then 
reduce this lower bound to $\theta \in \Theta_N$ via Bernstein's inequality.
Let 
\begin{align}
E = E(x_1, \ldots, x_K) \defeq \left\{\theta \in (0,2\pi): \abs{S_K(x_1, \ldots, x_K, \theta)} \ge \frac{1}{4} \sqrt{K \log K} \right\}.
\label{eqn:lower_E_def}
\end{align}
Let $\abs{E}$ be the measure of $E$.
We will lower bound $\abs{E}$ via the $2\lfloor \log_6 K \rfloor$-th moment of $S_K(\theta)$ assuming $\theta$ is drawn from $[0,2\pi]$ uniformly at random.
In the rest of Section \ref{sec:lower}, we assume every $\log$ term has base $6$.
Inspired by \cite{kac39},
we derive a formula for the moments of $S_K(\theta)$ and defer its proof to \ref{sec:appendix_lower}.

\begin{claim}
Let $p \le 4 \log K$ be a positive and even integer. Suppose $a_1, \ldots, a_K$ satisfy the assumption \eqref{eqn:lower_assumption} in Theorem \ref{thm:lower}. 
Let $A_K^p$ be the $p$th moment of $S_K(\theta)$. Then,
\begin{align}
A_K^p \defeq
\frac{1}{2\pi} \int_{0}^{2\pi} S_K(\theta)^p d\theta
= 2^{-p} \sum_{\substack{l \le p/2 \\ p_1 + \cdots + p_l = p \\ p_1,\ldots, p_l \text{ are even}}} \frac{p!}{\left((p_1/2)! \cdots (p_l/2)! \right)^2} \begin{pmatrix}
  K \\
  l
\end{pmatrix}.
\label{eqn:anp_form}  
\end{align}
\label{clm:lower_pth_moment}
\end{claim}

We want to further simplify the formula $A_K^p$ for $p = 2 \lfloor \log K \rfloor$ and $4 \lfloor \log K \rfloor$ 
when $K$ is sufficiently large.
Under these conditions, we observe the sum in the formula for $A_K^p$ (Equation \eqref{eqn:anp_form}) 
is dominated by the term where $l = p/2$ and $p_1 = \cdots = p_l = 2$.
This leads to the following lemma.

\begin{lemma}
Suppose $a_1, \ldots, a_K$ satisfy the condition in Theorem \ref{thm:lower}.
Then, for $p = 2 \lfloor \log K \rfloor, 4 \lfloor \log K \rfloor$,
\begin{align*}
\frac{A_K^p}{(K/2)^{p/2}} = (1 \pm o_K(1)) (p-1)!!,
\end{align*}
where $(p-1)!!$ is the double factorial.
\label{lem:limit}
\end{lemma}

To prove Lemma \ref{lem:limit}, we need the following formula for the double factorial.

\begin{claim}
For any positive and even integer $p$, $(p-1)!! = \frac{p!}{2^{p/2} (p/2)!}$.
\label{clm:double_factorial}
\end{claim}

\begin{proof}[Proof of Lemma \ref{lem:limit}]
We first show that $A_K^p$ is well approximately by $2^{-p} p! \begin{pmatrix}
  K \\
  p/2
\end{pmatrix}$.
That is, the sum in Equation \eqref{eqn:anp_form} can be approximated by the term in which $l = p/2$ and $p_1 = \ldots = p_l = 2$.
Define 
\[
\alpha_{K,p,l} \defeq p! \begin{pmatrix}
  K \\
  l
\end{pmatrix}  \sum_{\substack{ p_1 + \cdots + p_l = p \\ p_1,\ldots, p_l \text{ are even}}} 
\frac{1}{\left((p_1/2)! \cdots (p_l/2)! \right)^2}. 
\]
For any integer $1 \le l \le \frac{p}{2}-1$, let 
$$m_{p,l} \defeq \abs{\left\{p_1, \ldots, p_l \in \mathbb{Z}_+: p_1 + \cdots + p_l = p, p_1, \ldots, p_l \text{ are even} \right\}}.$$
Then,
\begin{align*}
\frac{\alpha_{K,p,l}}{\alpha_{K,p,p/2}} & = \frac{\begin{pmatrix} K \\l
\end{pmatrix}}{\begin{pmatrix} K \\p/2
\end{pmatrix}}
\cdot \sum_{\substack{p_1 + \cdots + p_l = p \\ p_1,\ldots, p_l \text{ are even}}} 
\frac{1}{\left((p_1/2)! \cdots (p_l/2)! \right)^2}  \\
& \le \frac{\begin{pmatrix} K \\ p/2-1
\end{pmatrix}}{\begin{pmatrix} K \\p/2
\end{pmatrix}}
\cdot m_{p,l} 
= \frac{p/2}{K-p/2+1} \cdot m_{p,l} 
\le \frac{\log K}{K - \log K} \cdot m_{p,l}.
\end{align*}
Take sum over all $1 \le l \le \frac{p}{2}-1$,
\begin{align*}
\frac{\sum_{l=1}^{p/2-1} \alpha_{K,p,l}}{\alpha_{K,p,p/2}}
\le \frac{\log K}{K - \log K} \sum_{l=1}^{p/2-1} m_{p,l}
\le \frac{\log K}{K - \log K} \sum_{l=1}^{p/2} m_{p,l}.
\end{align*}
The number $\sum_{l=1}^{p/2} m_{p,l}$ can be interpreted as the number of possible sets of integers whose sum equals $p/2$.
\[
  \sum_{l=1}^{p/2} m_{p,l} = 2^{p/2-1}
\le \frac{1}{2} K^{\log_6 4}.
\]
Thus,
\[
\frac{\sum_{l=1}^{p/2-1} \alpha_{K,p,l}}{\alpha_{K,p,p/2}}
\le \frac{\log K}{K - \log K} \cdot \frac{1}{2} K^{\log_6 4} 
\defeq \beta_K.
\]
That is, 
\begin{align*}
2^p A_{K}^p = (1 + \beta_K) \alpha_{K,p,p/2}.
\end{align*}
In addition, $2^p A_{K}^p \ge \alpha_{K,p,p/2}$. This implies 
\begin{align*}
\frac{A_K^{p}}{(K/2)^{p/2}} & = (1+\beta_K) \frac{\alpha_{K,p,p/2}}{(2K)^{p/2}} \\
& = (1 + \beta_K) (2K)^{-p/2} \cdot p! \cdot \frac{K (K-1) \cdots (K-p/2+1)}{(p/2)!}.
\end{align*}
By Claim \ref{clm:double_factorial},
\begin{align*}
\frac{A_K^p}{(K/2)^{p/2}} \cdot \frac{1}{(p-1)!!}
& = (1+\beta_K) \frac{ K (K-1) \cdots (K-p/2+1)}{K^{p/2}}
\le 1+\beta_K,
\end{align*}
and 
\begin{multline*}
\frac{A_K^p}{(K/2)^{p/2}} \cdot \frac{1}{(p-1)!!}
\ge \left(1 - \frac{p}{2K} \right)^{p/2}
\ge \exp(-p^2 / 2K)  \\
\ge 1 - \frac{p^2}{2K}
\ge 1 - \frac{8 \log^2 K}{K}.
\end{multline*}
Thus, $\frac{A_K^p}{(K/2)^{p/2}} = (1 \pm o_K(1)) (p-1)!!$.
\end{proof}

We will employ the following anti-concentration inequality from \cite{SZ54} to establish a lower bound for $\abs{E}$ defined in Equation \eqref{eqn:lower_E_def}.

\begin{lemma}[Lemma 4.2.4 of \cite{SZ54}]
Let $\psi(x) \ge 0$ where $a \le x \le b$ be a bounded real-valued function satisfying $\int_a^b \psi(x) dx \ge A > 0$ and $\int_a^b \psi^2(x) dx \le B$ (clearly, $A^2 \le B$). 
Let $0 < \delta < 1$ and $\abs{S}$ be the measure of a set $S = \{x: \psi(x) \ge \delta A\}$.
Then, $\abs{S} \ge (1-\delta)^2 \frac{A^2}{B}$.
\label{lem:sz54}
\end{lemma}

\begin{lemma}
Let $E$ be the set defined as in Equation \eqref{eqn:lower_E_def}.
Then, the measure of the set $\abs{E} > (1+o_K(1)) K^{-1}$.
\label{lem:lower_E}
\end{lemma}

\begin{proof}
Apply Lemma \ref{lem:sz54} with $\psi(\theta) = S_K(\theta)^p$ where $p = 2\lfloor \log K \rfloor$. 
Let 
\[
A \defeq \frac{1}{2\pi} \int_0^{2\pi} \psi(\theta) = A_K^p, ~
B \defeq \frac{1}{2\pi} \int_0^{2\pi} \psi(\theta)^{2} = A_K^{2p}.
\]
Take $\delta = K^{-1}$ and $S = \{\theta: \abs{\psi(\theta)} \ge K^{-1} A\}$.
Then,
\begin{align*}
\abs{S}
\ge (1- K^{-1})^2 \frac{A^2}{B}
\ge (1- K^{-2}) \frac{A^2}{B}.
\end{align*}
By Lemma \ref{lem:limit}, the above measure is at least
\begin{multline*}
  (1+o_K(1)) (1- K^{-2}) \frac{((p-1)!!)^2}{(2p-1)!!} 
  = (1+o_K(1))  \left( \frac{p!}{(p/2)! 2^{p/2}} \right)^2 \cdot \frac{p! 2^p}{(2p)!} \\
  \ge (1+o_K(1)) (\frac{1}{3})^{p/2} \cdot (\frac{1}{2})^{p/2} 
  = (1+o_K(1)) K^{-1}.
\end{multline*}
For every $\theta \in S$,
\begin{multline*}
  \frac{\abs{\sum_{k=1}^K x_k \cos a_k \theta}}{\sqrt{K/2}}
  \ge \frac{K^{-1 / p} A^{1/p}}{\sqrt{K/2}} 
  = (1+o_K(1)) K^{-1 / p} ((p-1)!!)^{1/p}  \\
  = (1+o_K(1)) K^{-1 / p} \left(
    \frac{p!}{2^{p/2} (p/2)!}
  \right)^{1/p} 
  \ge \frac{1}{4} \sqrt{p}.
\end{multline*}
This implies 
\[
\abs{\sum_{k=1}^K x_k \cos a_k \theta}
\ge \frac{1}{4} \sqrt{K\log K}.
\]
Thus, $\abs{E} \ge \abs{S} \ge (1+o_K(1))K^{-1}$.
\end{proof}

We are ready to prove Theorem \ref{thm:lower}.

\begin{proof}[Proof of Theorem \ref{thm:lower}]

Let $N \ge 32 a_K$ be the number of vertices in the circulant graph. Fix any $x_1, \ldots, x_K \in \{\pm 1\}$. Let 
\begin{align*}
& \theta^* \defeq \argmax_{\theta \in (0,2\pi)} \abs{\sum_{k=1}^K x_k \cos (a_k \theta)}, \\
& \theta = \argmin\left\{\abs{y - \theta^*}: y \in \Theta_N \right\}.
\end{align*}
By Bernstein's inequality (Theorem \ref{thm:bernstein}),
\begin{align*}
\abs{S_K(\theta^*)} - \abs{S_K(\theta)}
\le \abs{S_K(\theta^*) - S_K(\theta)} 
\le \frac{2\pi}{N} \cdot a_K \cdot \abs{S_K(\theta^*)}.
\end{align*}
Rearrange the above inequality,
\[
\abs{S_K(\theta)} \ge \left( 1 - \frac{2\pi a_K}{N} \right)  S_K(\theta^*) \ge \frac{4}{5} S_K(\theta^*).
\]
By Lemma \ref{lem:lower_E}, $S_K(\theta^*) \ge \frac{1}{4} \sqrt{K \log K}$. Thus, 
\[
\abs{S_K(\theta)} \ge \frac{1}{5} \sqrt{K \log K}.
\]

Assume by contradiction, $x_1, \ldots, x_K \in \{\pm 1\}$ indicates a partition of the generators that satisfies Equation \eqref{eqn:lower_matrix_goal} in Theorem \ref{thm:lower}. 
Then, by Claim \ref{clm:lower_degree}, 
\[
\abs{\sum_{k=1}^K x_k} \le \frac{1}{20} \sqrt{K \log K}.
\]
This implies 
\begin{align*}
\max_{\theta \in \Theta_N} \abs{\sum_{k=1}^K x_k \left( 1-\cos(a_k \theta) \right)}
\ge \max_{\theta \in \Theta_N} \abs{S_K(\theta)} - \frac{1}{20} \sqrt{K \log K}
\ge \frac{3}{20} \sqrt{K \log K}.
\end{align*}
Since $\sum_{k=1}^K 1 - \cos(a_k \theta) \le 2K$ for every $\theta$, we have 
\begin{align*}
\max_{\theta \in \Theta_N} \abs{\frac{\sum_{k=1}^K x_k \left( 1-\cos(a_k \theta) \right)}{\sum_{k=1}^K 1 - \cos(a_k \theta)}}
\ge \frac{3}{40} \sqrt{\frac{\log K}{K}},
\end{align*}
contradicting Equation \eqref{eqn:lower_matrix_goal}. Thus, the theorem statement holds.
\end{proof}

\subsection{Effective Resistances}

In this section, we upper bound the maximum effective resistance of the circulant graph 
defined in Theorem \ref{thm:lower}. Combining Theorem \ref{thm:lower} and Lemma \ref{lem:lower_effective_resistance},
we prove Theorem \ref{thm:intro_lower}.

\begin{lemma}
Let $K, a_1, \ldots, a_K$ be given positive integers satisfying the assumptions in Theorem \ref{thm:lower}.
Then, there exists $N_0 > 0$ such that for every integer $N > N_0$, 
the maximum effective resistance between any two endpoints of an edge in $X = X(\mathbb{Z}_N, \{\pm a_1, \ldots, \pm a_K\})$ is at most $O(1/K)$.
\label{lem:lower_effective_resistance}
\end{lemma}

We will need the following two claims, whose proofs are deferred to \ref{sec:appendix_lower}.

\begin{claim}
Let $X = X(\mathbb{Z}_N, \{\pm a_1, \ldots, \pm a_K\})$ be a circulant graph. Then, for any edge $(u,v)$, 
\[
\er(u,v) = \frac{1}{N} \sum_{j=1}^{N-1} \frac{1 - \cos(\frac{2\pi j (u-v) }{N})}{\sum_{k=1}^K 1 - \cos(\frac{2\pi j a_k}{N})}.
\]
\label{fac:prelim_circ_er}
\end{claim}

\begin{claim}
Fix $K$, $a_1, \ldots, a_K$, and $a \in \{a_1, \ldots, a_K\}$. 
For $\theta \in [0,2\pi]$, denote 
\[
  f(\theta) \defeq 1 - \cos(a \theta), ~
  g(\theta) \defeq \sum_{k=1}^K 1 - \cos(a_k \theta).
\] 
Then, for any $\theta_0 \in [0,2\pi]$ with $g(\theta_0) = 0$, the limit 
$\lim_{\theta \rightarrow \theta_0} \frac{f(\theta)}{g(\theta)}$ exists and is finite.
Define $h(\theta) \defeq \frac{f(\theta)}{g(\theta)}$ for $\theta \in [0,2\pi]$ with $g(\theta) \neq 0$,
and extend $h(\theta)$ to $[0,2\pi]$ by defining $h(\theta_0) \defeq \lim_{\theta \rightarrow \theta_0} \frac{f(\theta)}{g(\theta)}$
for $\theta_0 \in [0,2\pi]$ with $g(\theta_0) = 0$.
Then, 
\begin{align*}
\lim_{N \rightarrow \infty} \frac{1}{N} \sum_{\theta \in \Theta_N} h(\theta)
= \frac{1}{2\pi} \int_0^{2\pi} h(\theta) d\theta
\end{align*}
where $\Theta_N \defeq \{\frac{2\pi j}{N}: j=1,\ldots, N-1 \}$.
\label{clm:sum_to_integral}
\end{claim}

\begin{proof}[Proof of Lemma \ref{lem:lower_effective_resistance}]
Fix $K$ and $a_1, \ldots, a_K$.
By Claim \ref{fac:prelim_circ_er} and \ref{clm:sum_to_integral}, the maximum 
effective resistance between any two endpoints of an edge in $X$ is 
\begin{multline*}
\lim_{N \rightarrow \infty} \max_{a \in \{a_1, \ldots, a_K\}} \frac{1}{N} \sum_{\theta \in \Theta_N } \frac{1 - \cos a\theta}{\sum_{k=1}^K 1 - \cos a_k \theta} \\ =
\max_{a \in \{a_1, \ldots, a_K\}} \frac{1}{2\pi} \int_0^{2\pi} \frac{1 - \cos a\theta}{\sum_{k=1}^K 1 - \cos a_k \theta} dx.
\end{multline*}
Let $S$ be the set of $\theta$ satisfying $\sum_{k=1}^K \cos a_k \theta \ge c K$ for some constant $0 < c < 1$.
\begin{align*}
& \Pr \left( \theta \in S \right)  \\
= & \Pr \left(
  \left(\frac{\sum_{k=1}^K \cos a_k \theta}{\sqrt{K/2}}\right)^{2 \log K} \ge (2c^2 K)^{\log K} \right) \\
\le & \frac{(1+o(1)) (2\log K - 1)!!}{(2c^2 K)^{\log K}} 
\tag{by Markov inequality and Lemma \ref{lem:limit}} \\
= & (1+o(1)) \frac{(2\log K)!}{2^{\log K} (\log K)! \cdot (2c^2 K)^{\log K}}
\tag{by Claim \ref{clm:double_factorial}} \\
= & (1+o(1)) \left(\frac{\log K}{2c^2 K}\right)^{\log K}.
\end{align*}
Thus,
\begin{align*}
 & \max_{a \in \{a_1, \ldots, a_K\}} \frac{1}{2\pi} \int_0^{2\pi} \frac{1 - \cos a\theta}{\sum_{k=1}^K 1 - \cos a_k \theta} d\theta \\
\le & \frac{2}{(1-c) K} + \max_{a \in \{a_1, \ldots, a_K\}} \frac{1}{2\pi} \int_{S} \frac{1 - \cos a\theta}{\sum_{k=1}^K 1 - \cos a_k \theta} d\theta \\
\le & \frac{2}{(1-c) K} +  \frac{1}{2\pi} \int_{S} \frac{\max_{a \in \{a_1, \ldots, a_K\}} 1 - \cos a\theta}{\sum_{k=1}^K 1 - \cos a_k \theta} d\theta \\
\le & \frac{2}{(1-c) K} + \frac{1}{2\pi} \int_S d\theta \\
=  & O(\frac{1}{K}).
\end{align*}
\end{proof}

\begin{proof}[Proof of Theorem \ref{thm:intro_lower}]
Let $K_0$ be the integer in Theorem \ref{thm:lower}
For any integer $K > K_0$, let $a_1,\ldots,a_K$ be integers satisfying the assumption Equation \eqref{eqn:lower_assumption} in Theorem \ref{thm:lower}; let $h(K,a_1, \ldots,a_K) \ge \max\{32 a_K, N_0\}$ where $N_0$ is the integer in Lemma \ref{lem:lower_effective_resistance}.

Assume, by contradiction, for any constant $c > 0$, there exists a partition of $C$ into symmetric subsets $C_1 \cup C_2$ satisfying the condition in Theorem \ref{thm:intro_lower}. Then, 
\begin{align*}
- 2c \sqrt{\alpha \log K} \cdot L_{X} \pleq
L_{X(\mathbb{Z}_N, C_1)} - L_{X(\mathbb{Z}_N, C_2)} \pleq 2c \sqrt{\alpha \log K} \cdot L_{X}.
\end{align*}
By Lemma \ref{lem:lower_effective_resistance}, there exists a constant $c' > 0$ such that 
\[
-c' \sqrt{\frac{\log K}{K}} \cdot L_{X} \pleq
L_{X(\mathbb{Z}_N, C_1)} - L_{X(\mathbb{Z}_N, C_2)} 
\pleq c'  \sqrt{\frac{\log K}{K}} \cdot L_{X}.
\]
The above equation contradicts Theorem \ref{thm:lower} for $c' < \frac{1}{20}$.
\end{proof}

%% file: multi_dims.tex
\section{Generalizations to Partitioning Cayley Graphs}
\label{sec:multi_circ}

In this section, we generalize our algorithm for circulant graphs with arithmetic progression generators (Section \ref{sec:circ_ap}) to Cayley graphs of finite abelian groups $\mathbb{Z}_{n_1} \times \mathbb{Z}_{n_2} \times \ldots \times \mathbb{Z}_{n_d}$ whose generators are from products of multiple arithmetic progressions \footnote{Every finite abelian group is isomorphic to a direct product of cyclic groups.}.
Specifically, we consider two types of graph products of circulant graphs: Cartesian product and Tensor product.

Let $G_1 = (V_1, E_1)$ and $G_2 = (V_2, E_2)$ be two undirected graphs. 
The \emph{Cartesian product} of $G_1$ and $G_2$, denoted by $G_1 \square G_2$, has vertex set $V_1 \times V_2$
and edge set 
$$\{((a_1, a_2), (b_1, b_2)): a_1 = b_1, (a_2, b_2) \in E_2 \text{ or } (a_1, b_1) \in E_1, a_2 = b_2\}.$$
The \emph{tensor product} of $G_1$ and $G_2$, denoted by $G_1 \times G_2$, has vertex set $V_1 \times V_2$ and edge set 
$$\{((a_1, a_2), (b_1, b_2)): (a_1, b_1) \in E_1, (a_2, b_2) \in E_2\}.$$
Let $X_1 = X(\mathbb{Z}_{n_1}, C_1), \ldots, X_d = X(\mathbb{Z}_{n_d}, C_d)$ be circulant graphs. Then, 
$X_1 \square \cdots \square X_d$ is a Cayley graph with vertex set $\mathbb{Z}_{n_1} \times \cdots \times \mathbb{Z}_{n_d}$ and generators $\{s_h e_h: h \in [d], s_h \in C_h \}$ where $e_h$ is the elementary basis vector in $d$ dimensions.
$X_1 \times \cdots \times X_d$ is a Cayley graph with vertex set $\mathbb{Z}_{n_1} \times \cdots \times \mathbb{Z}_{n_d}$ and generators $C_1 \times \cdots \times C_d$.

For simplicity, we consider $n_1 = \cdots = n_d = n$ and $C_1 = \pm [k_1], \ldots, C_d = \pm [k_d]$. It is very likely that our results extend to arbitrary $n_h$'s and arithmetic progressions.

\begin{theorem}
Let $n,d$ be positive integers. 
For each $h\in [d]$, let $k_h \in \mathbb{Z}_+$ satisfy $k_h \le \lfloor \frac{n-1}{2} \rfloor$, and let $X_h = X(\mathbb{Z}_n, \pm [k_h])$ be the circulant graph.
\begin{enumerate}
\item Let $X = X_1 \square \cdots \square X_d$.
We can partition the edge set of $X$ into $E_1 \cup E_2$, in polynomial time, such that the two edge-induced subgraphs $G_1, G_2$ satisfy
\begin{align}
  - O\left(1/\sqrt{k_{\min}} \right) L_{X} \pleq L_{G_1} - L_{G_2} \pleq O\left(1/\sqrt{k_{\min}} \right) L_{X},
\label{eqn:multi_dim_goal_cartesian}
\end{align}
where $k_{\min} = \min_{h\in [d]} k_h$.

\item Let $X = X_1 \times \cdots \times X_d$.
We can partition the edge set of $X$ into $E_1 \cup E_2$, in polynomial time, such that the two edge-induced subgraphs $G_1, G_2$ satisfy
\begin{align}
  -\frac{O(1)^d}{\sqrt{K}} L_{X} \pleq L_{G_1} - L_{G_2} \pleq 
\frac{O(1)^d}{\sqrt{K}} L_{X},
\label{eqn:multi_dim_goal_tensor}
\end{align}
where $K = \prod_{h=1}^d k_h$.
\end{enumerate}
\label{thm:circ_high_dim}  
\end{theorem}

When $d = O(1)$, Equation \eqref{eqn:multi_dim_goal_tensor} matches the bound in \mss \ (Theorem \ref{thm:mss}) up to constant; under an additional condition that $k_{\min} = \Omega(\sum_{h=1}^d k_h / d)$, Equation \eqref{eqn:multi_dim_goal_cartesian} also matches the bound in \mss.

Without loss of generality, assume the graphs $X$ in Theorem \ref{thm:circ_high_dim} are connected.
Otherwise, each connected component can be viewed as a graph defined in the same theorem but with potentially different \( n_1, \ldots, n_d \), up to isomorphism.

For both the Cartesian product and the tensor product of the circulant graphs in Theorem \ref{thm:circ_high_dim}, we partition the generators based on the signs of the generators of each circulant graph obtained from Theorem \ref{thm:circ_ap}.
For every $h \in [d]$, let $y_{h,1}, \ldots, y_{h,k_h} \in \{\pm 1\}$ satisfy the conditions in Lemma \ref{lem:circ_cond_large_den} and \ref{lem:circ_cond_small_den} with $k = k_h$. Such signs can be found in polynomial time as shown in Section \ref{sec:circ_ap_combine}.
Let $y_{h,-1} = y_{h,1}, \ldots, y_{h,-k_h} = y_{h,k_h}$.
\begin{enumerate}
\item If $X = X_1 \square \cdots \square X_d$, we sign each generator $s_h e_h$ by $y_{h,s_h}$.

\item If $X = X_1 \times \cdots \times X_d$, we sign each generator $(s_1, \ldots, s_d)^\top \in (\pm [k_1]) \times \cdots \times (\pm [k_d])$ by $\prod_{h=1}^d y_{h,s_h}$.
\end{enumerate}

We need the following fact about the eigenvalues of the Laplacian matrix of Cayley graphs over $\mathbb{Z}_n^d$.
We omit the proofs, which are standard.

\begin{fact}
The eigenvalues and the corresponding orthonormal eigenvectors of the Laplacian matrix of $X(\mathbb{Z}_n^d, C)$ are:
\begin{align*}
\abs{C} - \sum_{s \in C} \omega^{j^\top s}, ~ 
v_{j_1} \otimes \cdots \otimes v_{j_d}, ~ \forall j = (j_1, \ldots, j_d)^\top \in \mathbb{Z}_n^d
\end{align*}
where $\omega = \exp(\frac{2\pi i}{n})$ and $v_{j_h}$'s are defined as in Fact \ref{fac:circ_eigs}.
\label{clm:multi_dim_eigs}
\end{fact}

\begin{proof}[Proof of Theorem \ref{thm:circ_high_dim}]
Let $\theta \defeq \frac{2\pi}{n}$.
First, we consider $X = X_1 \square \cdots \square X_d$.
By Fact \ref{clm:multi_dim_eigs}, 
\begin{align*}
\norm{L_X^{\dagger 1/2} (L_{G_1} - L_{G_2}) L_X^{\dagger 1/2}} & = \max_{j \in \mathbb{Z}_n^d \setminus \{0\}} \frac{\abs{\sum_{h=1}^d \sum_{s_h=1}^{k_h} y_{h, s_h} (1 - \cos(j_h s_h \theta))}}{\sum_{h=1}^d \sum_{s_h=1}^{k_h}  (1 - \cos(j_h s_h \theta))} \\
& \le \max_{j \in \mathbb{Z}_n^d \setminus \{0\}} \frac{\sum_{h=1}^d \abs{\sum_{s_h=1}^{k_h} y_{h, s_h} (1 - \cos(j_h s_h \theta))}}{\sum_{h=1}^d \sum_{s_h=1}^{k_h}  (1 - \cos(j_h s_h \theta))} \\
& = O \left( 1/\sqrt{k_{\min}} \right).
\end{align*}
The last inequality is by our choice of $y_{h,s_h}$'s.

Second, we consider $X = X_1 \times \cdots X_d$. 
By Fact \ref{clm:multi_dim_eigs},
\begin{align*}
\norm{L_X^{\dagger 1/2} (L_{G_1} - L_{G_2}) L_X^{\dagger 1/2}} & = \max_{j \in \mathbb{Z}_n^d \setminus \{0\}} \frac{\abs{\prod_{h=1}^d \sum_{s_h=1}^{k_h} y_{h, s_h} \left(1 - \cos(j_h s_h \theta) \right)}}{K - \prod_{h=1}^d \sum_{s_h=1}^{k_h} \cos(j_h s_h \theta)}.
\end{align*}
We split $\mathbb{Z}_n^d \setminus \{0\}$ into two subsets: Let $\alpha = 0.9$, and 
\[
Z_1 \defeq \left\{j \in \mathbb{Z}_n^d: \exists h \in [d] \text{ s.t. } j_h \theta \in \left(\frac{\alpha}{k}, 2\pi - \frac{\alpha}{k} \right)
  \right\}, ~ Z_2 \defeq (\mathbb{Z}_n^d \setminus \{0\}) \setminus Z_1.
\]

For each $j \in Z_1$. By 
Lemma \ref{lem:circ_mid_den}, the denominator is at least $K -  (1 - \frac{0.9^2}{96}) K  = \Omega(K)$.
By Lemma \ref{lem:circ_cond_large_den}, the numerator is at most $O(1)^d \sqrt{K}$.
Thus, Equation \eqref{eqn:multi_dim_goal_tensor} holds for $j \in Z_1$.

For each $j \in Z_2$. The denominator is at least
\begin{align*}
  K - \prod_{h=1}^d \sum_{s_h=1}^{k_h} \left(1 - \frac{(j_h s_h \theta)^2}{4}\right)
& \ge K \left(1 - \exp\left(-\frac{\theta^2}{12} \sum_{h=1}^d (k_h j_h )^2 \right)
\right).
\end{align*}
The numerator is at most
\[
O \left( \prod_{h=1}^d (j_h k_h \theta \log k_h)^2 \right).
\]
If $\sum_{h=1}^d (j_h k_h \theta)^2 \ge 1$, then the ratio is at most 
\begin{align*}
O \left(\frac{\prod_{h=1}^d \log^2 k_h}{K} \right) = O(K^{-1/2}).
\end{align*}
If $ \sum_{h=1}^d (j_h k_h \theta)^2 < 1$, then the ratio is at most 
\begin{align*}
O(1) \cdot \frac{\prod_{h=1}^d (j_h k_h \theta \log k_h)^2 }{K \cdot  \sum_{h=1}^d (j_h k_h \theta)^2 }
= O \left( \frac{\prod_{h=1}^d \log^2 k_h}{K} \right) = O(K^{-1/2}).
\end{align*}
Thus, Equation \eqref{eqn:multi_dim_goal_tensor} holds for $j \in Z_2$.
\end{proof}

%% file: appendix_miss_proofs.tex
\section{Missing Proofs in Section \ref{sec:prelim}}
\label{sec:appendix_prelim}

We include the proofs of the facts in Section \ref{sec:prelim} Preliminaries.

\begin{proof}[Proof of Fact \ref{fac:er_lower}]
Suppose $G$ has $n$ vertices.
We sum up all the effective resistances over $(u,v) \in E$:
\begin{multline*}
\sum_{(u,v)\in E} \er(u,v) = \sum_{(u,v)\in E} (\chi_u - \chi_v)^\top L_G^{\dagger} (\chi_u - \chi_v)
\\
=  \tr \left(  L_G^{\dagger} \sum_{(u,v)\in E}  (\chi_u - \chi_v) (\chi_u - \chi_v)^\top \right)
= \tr \left( L_G^{\dagger} L_G \right)
= n-1.
\end{multline*}
Then, 
\begin{align*}
\max_{(u,v) \in E} \er(u,v) \ge \frac{n-1}{nd / 2} = \frac{2}{d} \left( 1 - \frac{1}{n}
\right).
\end{align*}
\end{proof}

\begin{proof}[Proof of Fact \ref{fac:lower}]
Suppose $G$ has $n$ vertices.
We first lower bound the spectral norm of $L_{G_1} - L_{G_2}$.
Let $\norm{\cdot}_F$ be the Frobenius norm of a matrix.
\begin{align*}
\norm{L_{G_1} - L_{G_2}}_2^2 \ge \frac{1}{n} \tr((L_{G_1} - L_{G_2})^2)
= \frac{1}{n} \norm{L_{G_1} - L_{G_2}}_F^2 
\ge \frac{1}{n} \cdot nd = d.
\end{align*}
That is, there exists a unit vector $x \in \mathbb{R}^n$ such that $\abs{x^\top (L_{G_1} - L_{G_2}) x} \ge \sqrt{d}$.
In addition, $x^\top L_G x \le 2d$.
Thus, 
\begin{align*}
    \abs{x^\top (L_{G_1} - L_{G_2}) x} \ge \frac{1}{2\sqrt{d}} x^\top L_G x.
\end{align*}
\end{proof}

\begin{proof}[Proof of Fact \ref{fac:circ_eigs}]
For any $j = 0,1,\ldots,n-1$ and $k = 0,1,\ldots,n-1$, the $k$th entry of $L_X v_j$ is (we index the rows and columns of $L_X$ by $\mathbb{Z}_n$):
\begin{align*}
(L_X v_j)_k = \frac{1}{\sqrt{n}} \left( \abs{C} \omega^{k j} - \sum_{s \in C} \omega^{(k+s) j}
\right)
= \frac{\omega^{kj}}{\sqrt{n}} \left( \abs{C} - \sum_{s \in C} \omega^{sj}
\right) = \lambda_j v_{jk}.
\end{align*}
\end{proof}

\section{Missing Proofs in Section \ref{sec:circ_ap}}
\label{sec:appendix_circ_ap}

\subsection{Proof of Lemma \ref{lem:circ_mid_den}}

This section proves Lemma \ref{lem:circ_mid_den}, which lower bounds the denominator in Equation \eqref{eqn:circ_ap_goal} for $\theta \in \Theta_1$.

Let 
$$f_k(\hat{\theta}) \defeq \left( \sum_{s=1}^k \cos(s\hat{\theta}) \right)^2 + \left( \sum_{s=1}^k \sin(s\hat{\theta}) \right)^2.$$
By Equation \eqref{eqn:fkt_def}, 
\begin{align*}
\sum_{s=1}^k 1 -  \cos((a+sb) \theta)
= k - \sqrt{f_k(b \theta)}.
\end{align*}
Since $\theta \in \Theta_1$, we can write $b\theta = \hat{\theta} + 2\pi l$ where $\hat{\theta} \in \left(\frac{\alpha}{k}, 2\pi - \frac{\alpha}{k}
\right)$ and $l$ is an integer.
Then, $f_k(b\theta) = f_k(\hat{\theta})$.
Thus, to prove Lemma \ref{lem:circ_mid_den}, it suffices to prove the following lemma.

\begin{lemma}
Suppose $k \ge 2$. Then,
\[
f_k(\hat{\theta}) \le \left(1 - \frac{\alpha^2}{24} +  o(1)\right)k^2,
~ \forall \hat{\theta} \in \left(\frac{\alpha}{k}, 2\pi - \frac{\alpha}{k} \right). 
\]
\label{lem:fkt_bound_whole}
\end{lemma}

By symmetry and the fact
\[
f_k(\pi) = \left( \sum_{s=1}^k \cos(s \pi) \right)^2
\le 1,
\]
it suffices to prove the inequality in Lemma \ref{lem:fkt_bound_whole} for any $\hat{\theta} \in \left( 
\frac{\alpha}{k}, \pi \right)$.
We further divide the interval $\left(\frac{\alpha}{k}, \pi \right)$ into two sub-intervals 
$$I_1 = \left(\frac{\alpha}{k}, \frac{3}{k} \right], ~ I_2 = \left( \frac{3}{k}, \pi \right)$$
and bound $f_k(\hat{\theta})$ over these two sub-intervals separately.

\begin{claim}
The function $f_k(\hat{\theta})$ is monotone decreasing over $I_1$.    
\label{clm:fkt_monotone}
\end{claim}

\begin{proof}
Take the derivative of $f_k(\hat{\theta})$ w.r.t. $\hat{\theta}$:
\begin{align*}
  f_k'(\hat{\theta}) = - 2 \left(\sum_{s =1}^k \cos(s \hat{\theta}) \right) \left(\sum_{s =1}^k \sin(s \hat{\theta}) \cdot s \right)
  + 2 \left(\sum_{s =1}^k \sin(s \hat{\theta}) \right) \left(\sum_{s =1}^k \cos(s \hat{\theta}) \cdot s \right)
\end{align*}
We will show $f_k'(\hat{\theta}) \le 0$ for $\hat{\theta} \in [0,\frac{3}{k}]$ by induction on $k$.
When $k = 1$, 
\[
\frac{1}{2} f'_1(\hat{\theta}) =  \sin(\hat{\theta}) \cos(\hat{\theta}) - \cos(\hat{\theta}) \sin(\hat{\theta}) = 0.
\]
Assume the statement holds for $k$. Then,
\begin{align*}
& \frac{1}{2} (f'_{k+1}(\hat{\theta}) - f'_k(\hat{\theta})) \\
= & \sum_{s=1}^k s \cdot \cos(s\hat{\theta}) \sin((k+1)\hat{\theta})
+ \sum_{s=1}^k (k+1) \cdot \sin(s\hat{\theta}) \cos((k+1) \hat{\theta}) \\
\quad & - \sum_{s=1}^k s \cdot \sin(s\hat{\theta}) \cos((k+1) \hat{\theta})
- \sum_{s=1}^k (k+1) \cdot \cos(s\hat{\theta}) \sin((k+1)\hat{\theta}) \\
= & - \sum_{s=1}^k (k+1-s) \sin((k+1-s) \hat{\theta}) 
\\
= & - \sum_{s=1}^k s \cdot \sin(s \hat{\theta}).
\end{align*}
For $k\hat{\theta} \in [0, 3]$, we have 
\[
f'_{k+1} (\hat{\theta}) - f'_k (\hat{\theta}) \le 0.
\]
By induction hypothesis, $f'_{k+1}(\hat{\theta}) \le 0$ for $\hat{\theta} \in [0, \frac{3}{k+1}]$.
\end{proof}

\begin{claim}
For every $\hat{\theta} \in I_1$, the function value $f_k(\hat{\theta}) \le \left(1 - \frac{\alpha^2}{24} +  o(1)\right)k^2$.
\label{clm:fkt_interval1}
\end{claim}

\begin{proof}
By Claim \ref{clm:fkt_monotone}, $f_k(\hat{\theta}) \le f_k (\frac{\alpha}{k})$ for every $\hat{\theta} \in I_1$.
Note for $0 < x < \frac{\pi}{2}$, 
\[
0 < \cos(x) \le 1 - \frac{x^2}{2} + \frac{x^4}{24}, ~ 
0 < \sin(x) \le x.  
\]
We have 
\begin{multline*}
\max_{\hat{\theta} \in I_1 } f_k(\hat{\theta}) \le
f(\frac{\alpha}{k}) 
\le \left( \sum_{s=1}^k 1 - \frac{(s\alpha / k)^2}{2} + \frac{(s\alpha / k)^4}{24} \right)^2
  + \left( \sum_{s=1}^k s\alpha / k \right)^2 \\
 \le \left(
    1 - \frac{\alpha^2}{24} + o(1)
  \right) k^2.
\end{multline*}
\end{proof}

Next, we upper bound $f_k(\hat{\theta})$ for $\hat{\theta} \in I_2$.

\begin{claim}
For every $\hat{\theta} \in I_2$, the function value $f_k(\hat{\theta}) \le 0.8k^2$.
\label{clm:fkt_interval2}
\end{claim}

\begin{proof}
We will upper bound the sum of cosines and the sum of sines separately.
\begin{align*}
\sum_{s=1}^k \cos(s \hat{\theta})
= Re \left( \sum_{s=1}^k e^{i s\hat{\theta}} \right) 
= Re \left( \frac{e^{i \hat{\theta}} - e^{i(k+1)\hat{\theta}}}{1 - e^{i\hat{\theta}}}
\right) 
 = - \frac{1}{2} + \frac{ \sin((k+\frac{1}{2}) \hat{\theta} )}{2 \sin(\frac{\hat{\theta}}{2})}.
\end{align*}
For $\hat{\theta} \in I_2$,
\begin{align*}
\abs{\sum_{s=1}^k \cos(s\hat{\theta})} \le \frac{1}{2} + \frac{1}{2\sin(\frac{\hat{\theta}}{2})}
\le \frac{1}{2} + \frac{1}{2\sin(\frac{3}{2k})}
\le \frac{k}{1.8},
\end{align*}
where the last inequality holds since $k \sin(\frac{3}{2k}) \ge 0.9$ for $k \ge 1$.

Similarly, 
\begin{align*}
\sum_{s=1}^k \sin(s\hat{\theta}) 
& = \im \left( \frac{e^{i\hat{\theta}} - e^{i(k+1)\hat{\theta}}}{1 - e^{i\hat{\theta}}} \right)  = \frac{1 - \cos k\hat{\theta}}{2\tan \frac{\hat{\theta}}{2}} + \frac{\sin k\hat{\theta}}{2}.
\end{align*}
For $\hat{\theta} \in I_2$, 
\begin{align*}
   \tan \frac{\hat{\theta}}{2} \ge \tan \frac{3}{2k} \ge \frac{3}{2k}.
\end{align*}
So, 
\begin{align*}
-\frac{1}{2} \le \sum_{s=1}^k \sin(s\hat{\theta}) \le \frac{2k}{3} + \frac{1}{2}.
\end{align*}
Combining the bounds for the sum of cosines and the sum of sines together, we have 
\[
f_k(\hat{\theta}) \le k^2 \left(\frac{1}{1.8^2} + \frac{4}{9} + o(1) \right) \le 0.8k^2.
\]
\end{proof}

\begin{proof}[Proof of Lemma \ref{lem:fkt_bound_whole}]
Combine Claims \ref{clm:fkt_interval1} and \ref{clm:fkt_interval2}.
\end{proof}

\subsection{Proof of Claim \ref{clm:circ_small_den_goal}}

\begin{proof}[Proof of Claim \ref{clm:circ_small_den_goal}]
Let $M \defeq \cos(a\theta) \beta_1' + \sin(a\theta) \beta_2' + \beta_3'$.
Since 
\[
M = \sum_{s=1}^k 1 - \cos(a\theta + s\hat{\theta}) > 0,
\]
The equation in the claim statement 
is equivalent to 
\begin{align*}
& \abs{\cos(a\theta) \beta_1 + \sin(a\theta) \beta_2 + \beta_3} \le 13 \gamma M \\
\iff & \cos(a\theta) (\beta_1 - \gamma \beta_1') + \sin(a\theta) (\beta_2 - \gamma \beta_2') + (\beta_3 - \gamma \beta_3') \le 12 \gamma M, \\
& \cos(a\theta) (\beta_1 + \gamma \beta_1') + \sin(a\theta) (\beta_2 + \gamma \beta_2') + (\beta_3 + \gamma \beta_3') \ge -12 \gamma M
\end{align*}
By assumption $-\gamma \beta_3' \le \beta_3 \le \gamma \beta_3'$, it suffices to show 
\begin{align}
& \cos(a\theta) (\beta_1 - \gamma \beta_1') \le 2\gamma M, 
~ \cos(a\theta) (\beta_1 + \gamma \beta_1') \ge -2\gamma M, 
\label{eqn:circ_small_den_goal_cos} \\
& \sin(a\theta) (\beta_2 - \gamma \beta_2') \le 10 \gamma M, 
~ \sin(a\theta) (\beta_2 + \gamma \beta_2') \ge -10 \gamma M.
\label{eqn:circ_small_den_goal_sin}
\end{align}
For the inequalities in Equation \eqref{eqn:circ_small_den_goal_cos}, 
For the first inequality, if $\cos(a\theta) \ge 0$, then the inequality holds since $-\gamma \beta_1' \le \beta_1 \le \gamma \beta_1'$.
Now, consider $\cos(a\theta) < 0$. We have
\begin{align*}
\cos(a\theta) (\beta_1 - \gamma \beta_1') \le -\cos(a\theta) \cdot 2 \gamma \beta_1',
~ \cos(a\theta) (\beta_1 + \gamma \beta_1') \ge \cos(a\theta) \cdot 2 \gamma \beta_1'.
\end{align*}
So, to prove the inequalities in Equation \eqref{eqn:circ_small_den_goal_cos}, it suffices to show 
\[
2 \cos(a\theta) \beta_1' + \sin(a\theta) \beta_2' + \beta_3' \ge 0.
\]
Take the Taylor expansions for $\sin(s\hat{\theta})$ and $\cos(s\hat{\theta})$,
\begin{align*}
& 2 \cos(a\theta) \beta_1' + \sin(a\theta) \beta_2' + \beta_3' \\
\ge & \cos(a\theta) \cdot \frac{1}{3} (1 \pm \frac{\alpha^2}{20}) k^3 \hat{\theta}^2 + \sin(a\theta) \cdot \frac{1}{2}(1\pm \frac{\alpha^2}{3}) k^2 \hat{\theta} + k(1-\cos(a\theta)) \\
\ge & k \left( 1 - \sqrt{\left( 1 - \frac{1}{3}(1 +\frac{\alpha^2}{20}) k^2 \hat{\theta}^2 \right)^2 + \left( \frac{1}{2} (1 + \frac{\alpha^2}{3}) k\hat{\theta} \right)^2 }
\right) \tag{by the Cauchy-Schwarz inequality} \\
\ge & 0 \tag{since $\abs{k\hat{\theta}} \le \alpha \le 0.9$}
\end{align*}
A similar argument proves the inequalities in Equation \eqref{eqn:circ_small_den_goal_sin}.

\end{proof}

\section{Missing Proofs in Section \ref{sec:lower}}
\label{sec:appendix_lower}

\begin{proof}[Proof of Claim \ref{clm:lower_pth_moment}]
The $p$th moment $S_K(\theta)^p$ equals a sum of weighted \\ $\cos^{p_1}(a_{k_1} \theta) \cdots \cos^{p_l}(a_{k_l} \theta)$ where $p_1 + \cdots + p_l = p$.
For every integer $q$,
\begin{align*}
  \cos^{q}( a_{k} \theta) = \left( \frac{e^{i a_k \theta} + e^{-i a_k \theta}}{2} \right)^q
= e^{-iq a_k \theta} \cdot 2^{-q} \sum_{s=0}^q \begin{pmatrix}
  q \\
  s
\end{pmatrix} e^{i(2s a_k \theta)}.
\end{align*}
Thus, 
\begin{multline*}
\cos^{p_1} (a_{k_1} \theta) \cdots \cos^{p_l} (a_{k_l} \theta)
 = 2^{-p} \prod_{i=1}^l \sum_{s_i=0}^{p_i} \begin{pmatrix}
  p_i \\
  s_i
\end{pmatrix} e^{i(2s_i - p_i) a_{k_i} \theta} \\
= 2^{-p} \sum_{s_1,\ldots, s_l} \begin{pmatrix}
  p_1 \\
  s_1
\end{pmatrix} \cdots \begin{pmatrix}
  p_l \\
  s_l
\end{pmatrix} e^{i((2s_1-p_1) a_{k_1} + \cdots + (2s_l-p_l) a_{k_l} )\theta }.
\end{multline*}
Let $\beta \defeq (2s_1-p_1) a_{k_1} + \cdots + (2s_l-p_l) a_{k_l}$.
The integral 
\begin{align*}
  \frac{1}{2\pi} \int_0^{2\pi} e^{i\beta \theta } d\theta = \left\{
\begin{array}{ll}
  0, & \quad \text{ if } \beta \neq 0 \\
  1, & \quad \text{ otherwise } 
\end{array}
  \right.
\end{align*}
By assumption in Equation \eqref{eqn:lower_assumption}, $\beta = 0$
if and only if $s_1 = \frac{p_1}{2}, \ldots s_l = \frac{p_l}{2}$; this is because when some $s_{k_j} \neq \frac{p_j}{2}$,  
\begin{align*}
\abs{(2s_{k_1}-p_{k_1}) a_{k_1} + \cdots + (2s_{k_l}-p_{k_l}) a_{k_l}}
& \ge a_{k_l} - \sum_{i=1}^{l-1} \abs{2s_{k_i} - p_{k_i}} a_{k_i} \\
& \ge a_{k_l} - p (a_1 + \ldots + a_{k_l-1}) > 0.
\end{align*}
Thus,
\begin{align*}
  \frac{1}{2\pi} \int_{0}^{2\pi} S_K(\theta)^p d\theta
& = 2^{-p} \sum_{\substack{p_1 + \cdots + p_l = p \\ p_1,\ldots, p_l \text{ are even}}} \sum_{k_1 < \cdots < k_l} \frac{p!}{p_1! \cdots p_l!}
\begin{pmatrix}
  p_1 \\
  p_1/2
\end{pmatrix} \cdots \begin{pmatrix}
  p_l \\
  p_l/2
\end{pmatrix} \\
& = 2^{-p} \sum_{\substack{p_1 + \cdots + p_l = p \\ p_1,\ldots, p_l \text{ are even}}} \frac{p!}{p_1! \cdots p_l!} \begin{pmatrix}
  K \\
  l
\end{pmatrix} 
\begin{pmatrix}
  p_1 \\
  p_1/2
\end{pmatrix} \cdots \begin{pmatrix}
  p_l \\
  p_l/2
\end{pmatrix} \\
& = 2^{-p} \sum_{\substack{p_1 + \cdots + p_l = p \\ p_1,\ldots, p_l \text{ are even}}} \frac{p!}{\left((p_1/2)! \cdots (p_l/2)! \right)^2} \begin{pmatrix}
  K \\
  l
\end{pmatrix}. 
\end{align*}
\end{proof}

\begin{proof}[Proof of Claim \ref{fac:prelim_circ_er}]
The pseudo-inverse of the Laplacian matrix of $X$ is $L_X^{\dagger} = \sum_{j=1}^{N-1} \lambda_j^{-1} v_j v_j^\top$,
where $\lambda_j$'s and $v_j$'s are defined as in Fact \ref{fac:circ_eigs}.
The effective resistance between $u$ and $v$ is
\begin{align*}
\er(u,v) & = (\chi_u - \chi_v)^\top  L_X^{\dagger} (\chi_u - \chi_v)  \\
 & = \frac{1}{N} \sum_{j=1}^{N-1} \lambda_j^{-1} \left( \omega^{uj} - \omega^{vj}
  \right)^2 \\
& = \frac{1}{N} \sum_{j=1}^{N-1} \lambda_j^{-1} \left( \omega^{(u-v)j} - 1
  \right)^2 \\
& = \frac{1}{N} \sum_{j=1}^{N-1} \lambda_j^{-1} 
\left( 2 - 2 \cos(\frac{2\pi j(u-v)}{N})
\right) \\
 & = \frac{1}{N} \sum_{j=1}^{N-1} \frac{1 - \cos(\frac{2\pi j (u-v)}{N})}{\sum_{k=1}^K 1 - \cos(\frac{2\pi j a_k}{N})}.
\end{align*}
\end{proof}

\begin{proof}[Proof of Claim \ref{clm:sum_to_integral}]
We first show that for every $\theta_0 \in [0,2\pi]$ with $g(\theta_0) = 0$, 
$\lim_{\theta \rightarrow \theta_0} \frac{f(\theta)}{g(\theta)}$ exists and is finite.
Since $g(\theta_0) = 0$ and $g(\theta_0)$ is a sum of non-negative terms, 
every term $1 - \cos(a_k \theta_0) = 0$ for every $k \in \{1,\ldots, K\}$.
\begin{multline*}
\lim_{\theta \rightarrow \theta_0} \frac{f(\theta)}{g(\theta)}
= \lim_{\Delta \rightarrow 0} \frac{1 - \cos(a\theta_0 + a\Delta)}{\sum_{k=1}^K 1 - \cos(a_k \theta_0 + a_k \Delta)} \\
= \lim_{\Delta \rightarrow 0} \frac{\sum_{j=1}^{\infty} \frac{(-1)^{j+1}}{(2j)!} (a\Delta)^{2j}}{\sum_{k=1}^K \sum_{j=1}^{\infty} \frac{(-1)^{j+1}}{(2j)!} (a_k \Delta)^{2j}}
= \frac{a^2}{\sum_{k=1}^K a_k^2}. 
\end{multline*}

We next prove the second part of the claim.
Since $f,g$ are continuous, the function $h$ is continuous at every point $\theta \in [0,2\pi]$ with $g(\theta) \neq 0$.
By our definition of $h(\theta)$ for $g(\theta) = 0$, we can check that $h$ is continuous over $[0,2\pi]$.
This implies that $h$ is uniformly continuous over $[0,2\pi]$.
Therefore, 
\[
\lim_{N \rightarrow \infty} \frac{1}{N} \sum_{\theta \in \Theta_N} h(\theta) = \frac{1}{2\pi} \int_{0}^{2\pi} h(\theta) d\theta
- \lim_{N \rightarrow \infty} \frac{h(0)}{N} =  \frac{1}{2\pi} \int_{0}^{2\pi} h(\theta) d\theta.
\]
\end{proof}